\documentclass{amsart}
\pdfoutput=1
\usepackage{amssymb}
\usepackage{microtype}
\usepackage{tikz}
\usetikzlibrary{arrows}
\usepackage{booktabs}
\usepackage[pdftex,colorlinks,citecolor=black,linkcolor=black,urlcolor=black,bookmarks=false]{hyperref}

\def\arXiv#1{arXiv:\href{http://arXiv.org/abs/#1}{#1}}
\usepackage{doi}
\usepackage{subfigure}
\usepackage{bbm} % For mathbbm{1}

\newtheorem{theorem}{Theorem}[section]
\newtheorem{proposition}[theorem]{Proposition}
\newtheorem{lemma}[theorem]{Lemma}
\newtheorem{corollary}[theorem]{Corollary}

\theoremstyle{definition}

\numberwithin{figure}{section}
\numberwithin{equation}{section}
\numberwithin{table}{section}

\newcommand{\Z}{\mathbb{Z}}
\newcommand{\R}{\mathbb{R}}
\newcommand{\C}{\mathbb{C}}
\newcommand{\F}{\mathbb{F}}

\newcommand{\PGL}{\operatorname{PGL}}

\newcommand{\legendre}[2]{\genfrac{(}{)}{}{}{#1}{#2}}

\title{Non-planarity of Markoff graphs mod $p$}

\author{Matthew de Courcy-Ireland}
\address{Institute of Mathematics\\
EPFL\\
CH-1015 Lausanne, Switzerland} \email{matthew.decourcy-ireland@epfl.ch}
\date{August 23, 2022}
\begin{document}

\begin{abstract}
We prove the non-planarity of a family of 3-regular graphs constructed from the solutions to the Markoff equation $x^2+y^2+z^2=xyz$ modulo prime numbers greater than 7. The proof uses Euler characteristic and an enumeration of the short cycles in these graphs. Non-planarity for large primes would follow assuming a spectral gap, which was the original motivation. For primes congruent to 1 modulo 4, or congruent to 1, 2, or 4 modulo 7, explicit constructions give an alternate proof of non-planarity. 
\end{abstract}

\maketitle

\section{Introduction} \label{sec:introduction}

For each prime number $p$, we consider a graph whose vertices are triples
in $\F_p^3$, with edges connecting a vertex $(x,y,z)$ to
\begin{align*}
m_1(x,y,z) &= (yz-x,y,z) \\
m_2(x,y,z) &= (x,xz-y,z) \\
m_3(x,y,z) &= (x,y,xy-z)
\end{align*}
The operations $m_1, m_2, m_3$ preserve the polynomial $x^2+y^2+z^2 - xyz$. 
Thus the graph is a disjoint union of subgraphs corresponding to solutions of a Markoff-type equation
\[
x^2 + y^2 + z^2 = xyz + k
\]
with $k \in \F_p$. An especially interesting case is $k=0$, which Markoff investigated (over $\Z$ rather than $\F_p$) and found to be related to quadratic forms and Diophantine approximation \cite{M}.
By ``the Markoff graph mod $p$", we mean the graph with vertices $(x,y,z) \neq (0,0,0)$ satisfying $x^2+y^2 + z^2 = xyz$ in $\F_p$, and edges given by $m_1, m_2, m_3$ as above.
For example, Figure~\ref{fig:mod7} shows the Markoff graph mod 7. 

\begin{theorem} \label{thm:main}
The Markoff graph mod $p$ is planar if and only if the prime $p$ is $2$, $3$, or $7$.
\end{theorem}

In other words, for $p \neq 2, 3, 7$, these graphs cannot be drawn in the plane without some edges crossing.
This is an indirect test of the hypothesis that the Markoff graphs form an expander family as $p \rightarrow \infty$. Indeed, by the planar separator theorem of Lipton and Tarjan \cite{LT}, expansion is impossible in planar graphs.
As a proof of expansion in the Markoff family remains elusive, we became interested in finding a direct proof that they are not planar.
We recall this connection in Section~\ref{sec:lt}.
We refer to \cite{KS} for more on the spectral properties of planar graphs.

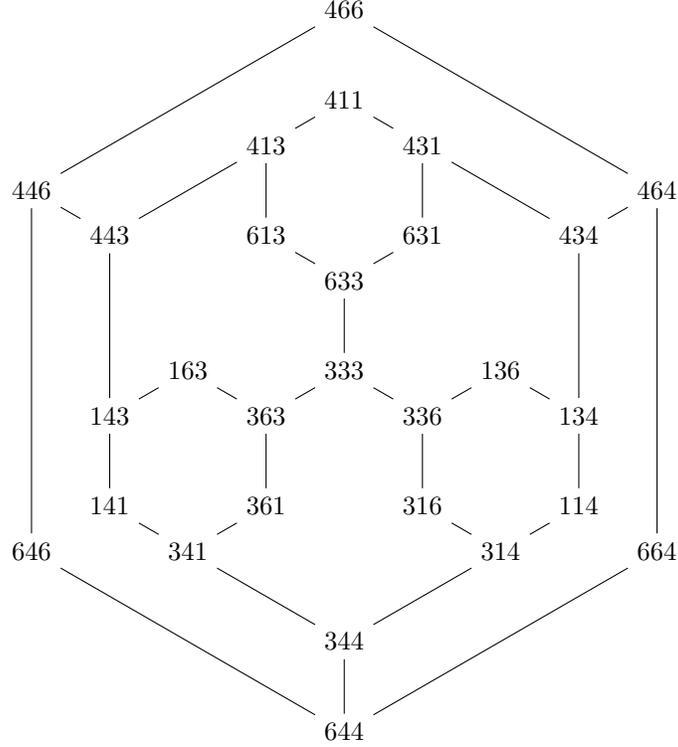
\begin{figure} 
\begin{tikzpicture}[scale=1.2]
\pgfmathsetmacro{\h}{0.866025}
\draw (0,0) node(333){333};
\draw (0,1) node(633){633};
\draw (\h,-1/2) node(336){336};
\draw (-\h,-1/2) node(363){363};
\draw (\h,3/2) node(631){631};
\draw (-\h,3/2) node(613){613};
\draw (\h,5/2) node(431){431};
\draw (-\h,5/2) node(413){413};
\draw (0,3) node(411){411};
\draw (-2*\h,0) node(163){163};
\draw (2*\h,0) node(136){136};
\draw (-\h,-3/2) node(361){361};
\draw (\h,-3/2) node(316){316};
\draw (3*\h,-1/2) node(134){134};
\draw (-3*\h,-1/2) node(143){143};
\draw (3*\h,-3/2) node(114){114};
\draw (-3*\h,-3/2) node(141){141};
\draw (-2*\h,-2) node(341){341};
\draw (2*\h,-2) node(314){314};
\draw (3*\h,3/2) node(434){434};
\draw (-3*\h,3/2) node(443){443};
\draw (4*\h,-2) node(664){664};
\draw (0,-3) node(344){344};
\draw (0,-4) node(644){644};
\draw (-4*\h,2) node(446){446};
\draw (4*\h,2) node(464){464};
\draw (0,4) node(466){466};
\draw (-4*\h,-2) node(646){646};
\draw (333)--(633);
\draw (333)--(363);
\draw (333)--(336);
\draw (466)--(446)--(646)--(644)--(664)--(464)--(466);
\draw (411)--(413)--(613)--(633)--(631)--(431)--(411);
\draw (141)--(341)--(361)--(363)--(163)--(143)--(141);
\draw (114)--(134)--(136)--(336)--(316)--(314)--(114);
\draw (143)--(443)--(413);
\draw (431)--(434)--(134);
\draw (341)--(344)--(314);
\draw (443)--(446);
\draw (464)--(434);
\draw (644)--(344);
\end{tikzpicture}
\caption{The Markoff graph mod 7 is planar. The vertices are the 28 solutions to $x^2+y^2+z^2=xyz \bmod 7$, excluding $(0,0,0)$. The labels abbreviate $(x,y,z)$ by $xyz$, with edges corresponding to the moves $x \mapsto yz - x$, $y \mapsto xz-y$, and $z \mapsto xy-z$. To obtain a 3-regular graph, small loops can be drawn at the vertices of degree 2 without crossing any other edges. These vertices are fixed by a move on one of the coordinates, as for instance $(1,6,3)$ is fixed by changing 3 to $1 \times 6 - 3 = 3$.
}
\label{fig:mod7}
\end{figure}

The intuition behind the proof is that a planar graph cannot have too many edges. The following folklore lemma can be shown using Euler characteristic, as we review in Section~\ref{sec:euler}.
\begin{lemma} \label{lem:planar}
If a planar connected graph has $V$ vertices, $E$ edges, and no cycles of length less than $g$, then
\begin{equation} \label{eqn:planar-ev}
E \leq \frac{g}{g-2}(V-2)
\end{equation}
\end{lemma}
For a graph with 3 edges at every vertex and no self-edges, it must be that $E=3V/2$. If there are no cycles of length less than $g=6$, then equation (\ref{eqn:planar-ev}) is absurd:
\[
\frac{3V}{2} = E \leq \frac{6}{6-2} (V-2) < \frac{3V}{2}
\]
from the strict inequality $V-2 < V$. This shows that a finite 3-regular graph of girth 6 cannot be planar. For comparison, there is an infinite 3-regular graph of girth 6, given by tiling the plane with hexagons. Attempts to truncate this infinite graph must introduce either crossings between edges, or cycles of length less than 6, or vertices of degree different from 3. 

The proof of Theorem~\ref{thm:main} applies the same logic to the Markoff graphs mod $p$. The number of edges is not quite $3V/2$ because of a small number of self-edges whenever $(x,y,z)$ is fixed by one of the Markoff moves $m_1, m_2, m_3$. This occurs for instance at (1,6,3) in the Markoff graph mod 7, with $3 = 6 \times 1 - 3$.
Moreover, there can be cycles of length shorter than 6. We will see that these are rare, and an approximate version of (\ref{eqn:planar-ev}) still yields a contradiction for sufficiently large $p$ provided we take $g=7$ rather than $g=6$, to compensate for these self-edges and short cycles.
The inequalities leave only two cases unsettled, $p=11$ and $p=13$, whose non-planarity can be shown directly to complete the proof.

The same approach gives a bound on the Euler characteristic that would be needed for a surface to accommodate the Markoff graph mod $p$. To make sense of the statement, recall that the Euler characteristic of a surface is typically negative. 
\begin{theorem} \label{thm:euler-char}
If $\chi$ is the Euler characteristic of a surface in which the Markoff graph mod $p$ can be embedded, then as $p \rightarrow \infty$,
\[
\left( \frac{1}{2} - o(1) \right)p^2 \leq -\chi
\]
\end{theorem}
Given $\chi$, Theorem~\ref{thm:euler-char} shows that there are only finitely many primes for which the Markoff graph mod $p$ can be embedded in a surface of that Euler characteristic. Indeed, such an embedding is impossible for $p > (1+o(1))\sqrt{2|\chi|}$, although our estimates on the term $o(1)$ are somewhat impractical. 
Theorem~\ref{thm:main} is a more precise statement of this form with $\chi=2$ for the planar case.
The order of magnitude $p^2$ in Theorem~\ref{thm:euler-char} is correct: drawing each edge on a handle of its own gives an embedding in a surface with $-\chi = (3+o(1)) p^2$.

The exceptions $p=2,3,7$ in Theorem~\ref{thm:main} give the smallest Markoff graphs.
The number of vertices in the Markoff graph for an odd prime $p$ is $p^2 + 3p(-1)^{(p-1)/2}$, by a formula of Carlitz \cite{Carlitz}, which we review in Lemma~\ref{lem:counts}.
In particular, there are 28 vertices for $p=7$ compared to 40 for $p=5$.
For $p=3$, there are no solutions to $x^2+y^2+z^2=xyz$ besides $(0,0,0)$, connected to itself by all three moves $m_1$, $m_2$, $m_3$, so the Markoff graph mod 3 is empty.
However, one can obtain a more interesting example mod 3 from the rescaling $x^2+y^2+z^2=3xyz$, as we describe in the conclusion.
For $p=2$, there are four non-zero solutions, namely $(1,1,1)$ connected to the permutations of $(0,1,1)$ with a pair of self-edges at each of the latter (Figure~\ref{fig:mod2}).
These more singular examples can be drawn in the plane, but it is natural to exclude them and think of the Markoff graph mod 7 as the only non-trivial planar example.

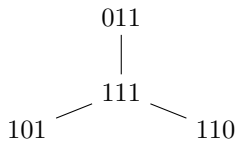
\begin{figure}
\begin{tikzpicture}
\draw (0,0) node(111){111};
\draw (111)++(90:1) node(011){011};
\draw (111)++(210:1) node[left](101){101};
\draw (111)++(330:1) node[right](110){110};
\draw (111)--(011);
\draw (111)--(101);
\draw (111)--(110);
\end{tikzpicture}
\caption{The Markoff graph for $p=2$, with $(0,1,1)$ fixed by the two moves sending either coordinate 1 to $0 \times 1 - 1 = 1 \bmod 2$.}
\label{fig:mod2}
\end{figure}

A famous theorem of Wagner and Kuratowski \cite{W,K} gives another approach to non-planarity, which is useful for the finite number of primes that remain after the main strategy is executed. Their theorem characterizes planar graphs in terms of the minimal obstructions: a graph is planar if and only if it does not contain any copies of the complete bipartite graph $K_{3,3}$ or the complete graph $K_5$ (with different notions of ``copy" in the exact formulations of Wagner and Kuratowski, as we review below). 
For example, we can prove the following theorems by finding explicit copies of $K_{3,3}$ inside the Markoff graph for certain primes $p$. 

\begin{theorem} \label{thm:mod4}
The Markoff graph mod $p$ is not planar for any prime number congruent to $1 \bmod 4$.
\end{theorem}

\begin{theorem} \label{thm:7}
If $-7$ is a non-zero quadratic residue modulo $p$, then the Markoff graph mod $p$ is not planar.
\end{theorem}

Notably, Theorem~\ref{thm:7} does not apply when $p=7$, and the Markoff graph is planar in that case.
By quadratic reciprocity, $-7$ is a square modulo $p$ if and only if $p$ is a square modulo 7, that is, $p$ is 1, 2, or 4 modulo 7. 
Together, Theorems ~\ref{thm:mod4} and \ref{thm:7} have the following corollary, which combines the conditions modulo 4 and modulo 7 into different possibilities modulo 28.
\begin{corollary} \label{cor:28}
The Markoff graph mod $p$ is not planar for any odd prime $p \neq 7$, except possibly for $p \equiv 3, 19, 27 \bmod 28$. 
\end{corollary}
In terms of density, these constructions show that the Markoff graphs are non-planar for at least a fraction 3/4 of primes.
We will use them especially to show non-planarity for $p=11$ and $p=13$, which are the last cases remaining in the proof of Theorem~\ref{thm:main} after non-planarity for large $p$ has been achieved by the strategy of Section~\ref{sec:euler}. 

The method of proof is to find a copy of the complete bipartite graph on 3 pairs of vertices. 
The example for Theorem~\ref{thm:mod4} uses special solutions available only when $-1$ has a square root modulo $p$, in particular the lines contained in the Markoff cubic surface, while Theorem~\ref{thm:7} requires a square root of $-7$. 
These configurations are drawn in Figures~\ref{fig:1mod4} and \ref{fig:-7square}.
An interesting difference is that Theorem~\ref{thm:7} is local in nature: it involves paths of bounded length, whereas Theorem~\ref{thm:mod4} involves paths of length growing with $p$.

There is a subtle difference between the formulations of Wagner and Kuratowski, even though both lead to equivalent characterizations of planarity. In Kuratowski's theorem, a ``copy" is simply a subdivision of $K_{3,3}$, where each edge of $K_{3,3}$ is given by a path between its endpoints in the graph of interest. The Markoff graphs are 3-regular, so that $K_5$ cannot occur as a subdivision. This differs from Wagner's formulation, where a ``copy" refers to a graph minor. To show a graph is non-planar using Wagner's theorem,  $K_{3,3}$ or $K_5$ may be formed by contracting edges, as well as deleting edges or isolated vertices. 
Contracting an edge removes it and merges its endpoints into a single vertex, which allows $K_5$ to occur as a minor even for graphs with only 3 edges incident to each vertex. This is illustrated in Figure~\ref{fig:wagner}.
Looking for copies of $K_5$ might allow more flexibility in proving non-planarity, but one knows from Kuratowski's theorem that there must also be a subdivision of $K_{3,3}$ whenever $K_5$ occurs as a minor. 
In this sense, $K_{3,3}$ is the only obstruction to planarity for the Markoff graphs.

\begin{figure}
\begin{tikzpicture}[scale=1.2]
\foreach \i in {0,...,9}
\draw (\i*36:1) node[circle,fill=black,inner sep=1.5](\i){};
\draw[dashed] (1)--(2);
\draw[dashed] (3)--(4);
\draw[dashed] (5)--(6);
\draw[dashed] (7)--(8);
\draw[dashed] (9)--(0);
\draw (0)--(1);
\draw (2)--(3);
\draw (4)--(5);
\draw (6)--(7);
\draw (8)--(9);
\draw (0)--(3);
\draw (1)--(8);
\draw (2)--(5);
\draw (4)--(7);
\draw (6)--(9);
\end{tikzpicture}
\hspace{1cm}
\begin{tikzpicture}[scale=1.2]
\foreach \i in {0,...,4}
\draw (270+\i*72:1) node[circle,fill=black,inner sep=1.5](\i){};%{$\bullet$};
\draw (0)--(1)--(2)--(3)--(4)--(0);
\draw (0)--(2)--(4)--(1)--(3)--(0);
\end{tikzpicture}
\hspace{1cm}
\begin{tikzpicture}[scale=1.2]
\draw (0:1) node[circle,fill=lightgray,inner sep=1.5](0){$\bullet$};
\draw (72:1) node[circle,fill=lightgray,inner sep=1.5](2){$\bullet$};
\draw (144:1) node[circle,fill=lightgray,inner sep=1.5](4){$\bullet$};
\draw (36:1) node[circle,draw=black,inner sep = 1.5](1){$\bullet$};
\draw (108:1) node[circle,draw=black,inner sep=1.5](3){$\bullet$};
\draw (180:1) node[circle,draw=black,inner sep=1.5](5){$\bullet$};
\draw (0)--(1)--(2)--(3)--(4)--(5);
\draw (0)--(3);
\draw (2)--(5);
\draw (6*36:1) node[circle,fill=black,inner sep=1.5](6){};
\draw (7*36:1) node[circle,fill=black,inner sep=1.5](7){};
\draw (8*36:1) node[circle,fill=black,inner sep=1.5](8){};
\draw (9*36:1) node[circle,fill=black,inner sep=1.5](9){};
\draw (0)--(9)--(6)--(5);
\draw (1)--(8)--(7)--(4);
\end{tikzpicture}
\caption{The 3-regular graph on the left contains a copy of $K_5$ as a graph minor, obtained by contracting the dashed edges. A subdivision of $K_{3,3}$ for the same graph is shown at right.}
\label{fig:wagner}
\end{figure}
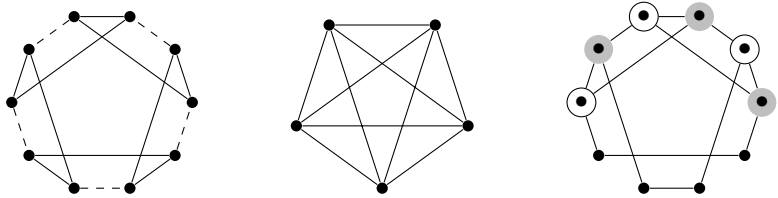

We exclude the point $(0,0,0)$ from the Markoff graph mod $p$ because it is fixed by all of $m_1, m_2, m_3$. The rest of the level set $x^2+y^2+z^2-xyz=0$ seems to form a connected graph. This was conjectured in Baragar's thesis \cite{Bar} and connectedness is now known for all sufficiently large primes $p$. Bourgain-Gamburd-Sarnak \cite{BGS} have been able to prove connectedness for many primes $p$ by a method that succeeds unless $p^2-1$ has an unusually large number of factors. In any case, their method shows that there is a ``giant component": for any $\varepsilon > 0$, once $p$ is large enough depending on $\varepsilon$, the Markoff graph has a connected component containing all but $O(p^{\varepsilon})$ vertices. On the other hand, Chen \cite{C} has shown that all connected components have size divisible by $p$. As a result, the giant component must coincide with the entire graph once $p$ is large enough.
An explicit threshold for how large $p$ should be has been determined by Fuchs, Litman, and Tran and communicated to the author: $p > 3 \times 10^{27}$. We also refer to their article with Lauter \cite{FLLT} for further analysis of the Markoff graphs, and cryptographic applications.

Numerical evidence obtained in \cite{dci-lee} suggests that the Markoff graphs are not only connected, but moreover form an expander family as $p \rightarrow \infty$. This seems to demand new techniques beyond what is involved in proving connectedness.
Non-planarity is a simple consequence of expansion that can be established more easily.
This in turn provides some indirect evidence in favour of expansion.

The rest of this article pursues these ideas in the following sequence.
In Section~\ref{sec:review}, we review how many vertices and edges are in the Markoff graph mod $p$, and some other basic parameters. 
In Section~\ref{sec:euler}, we outline the strategy leading to Theorem~\ref{thm:main}, recall the proof of Lemma~\ref{lem:planar}, and prove Theorem~\ref{thm:euler-char}. 
Sections~\ref{sec:words},~\ref{sec:323121},~\ref{sec:321321}, and \ref{sec:alt} complete the proof of Theorem~\ref{thm:main} by determining the fixed points of some short words in the Markoff moves $m_1$, $m_2$, $m_3$.
In Section~\ref{sec:cage}, we show that even in the hypothetical cases where the Markoff graph is not connected, the foregoing arguments show non-planarity of the giant component of Bourgain-Gamburd-Sarnak.
This relies on a lower bound for Euler's function $\phi(n)$, detailed in Section~\ref{sec:totient}.

In Section~\ref{sec:proof-mod4}, we prove Theorem~\ref{thm:mod4} on non-planarity for primes congruent to 1 mod 4, which takes advantage of lines contained in the Markoff cubic surface.
Section~\ref{sec:proof7} proves Theorem~\ref{thm:7}, which applies to some primes congruent to 3 mod 4 but not all. 
Section~\ref{sec:lt} reviews the Lipton-Tarjan theorem and its consequence that expansion cannot occur in planar graphs, which was our motivation for investigating the question of planarity.
We give a simple calculation that, based on the level of expansion observed numerically, estimates how large $p$ must be for this method to imply non-planarity.
We conclude with some examples in Sections~\ref{sec:examples} and \ref{sec:conc}, drawing the Markoff graphs for $p=5$ and 11, with an alternative scaling for $p=3$.

We recommend \cite{A} as an excellent account of the Markoff surface, and cite just a few examples of recent work in addition to \cite{BGS, C, FLLT} already discussed above.
The permutations generated by $m_1$, $m_2$, $m_3$ on solutions mod $p$ have been studied in \cite{CGMP, MPC}.
Over $\Z$, see \cite{MM} for recent work on the fractals introduced by Markoff in Diophantine approximation,
\cite{AD} for generalizations to modular billiards, and \cite{Mirz} for connections with hyperbolic geometry.

\section{Some key counts} \label{sec:review}

In this section, we record some of the fundamental counts to do with the Markoff graph mod $p$. How many vertices? edges? short cycles? Recall that the vertices of the graph are triples $(x,y,z) \neq (0,0,0)$ satisfying $x^2+y^2+z^2 =xyz \bmod p$. 
Some of the counting is best thought of more generally for surfaces of the form $x^2+y^2+z^2=xyz+k$, the case $k=0$ being somewhat degenerate.

\begin{lemma} \label{lem:counts}
\noindent
\begin{enumerate}
\item
The number of vertices in the Markoff graph mod $p$ is $p^2+3p$ if $p \equiv 1 \bmod 4$, or $p^2 - 3p$ if $p \equiv 3 \bmod 4$. 
\item
The Markoff graph mod $p$ is $3$-regular, except for $3(p-3)$ vertices if $p \equiv 3 \bmod 4$ or $3(p-5)$ vertices if $p \equiv 1 \bmod 4$, which each have two neighbours and a single self-edge.
\item
If $p \equiv 1 \bmod 4$, then the cycles of length $4$ in the Markoff graph mod $p$ are of the form
%\begin{center}
%\begin{tikzpicture}[scale=2.5]
%\draw (0,0)--(1,0)--(1,1)--(0,1)--(0,0);
%\draw (0,0) node[below,left]{$(x,y,0)$};
%\draw (1,0) node[below,right]{$(-x,y,0)$};
%\draw (1,1) node[above,right]{$(-x,-y,0)$};
%\draw (0,1) node[above,left]{$(x,-y,0)$};
%\draw (0.5,0.5) node{\small{$x^2+y^2=0$}};
%\end{tikzpicture}
%\end{center}
shown in Figure~\ref{fig:square-hex} with $z=0$, or similarly with $x=0$ or $y=0$. In total, there are $3(p-1)/2$ cycles of length $4$.
Ignoring self-edges, there are no shorter cycles.
\item
If $p \equiv 3 \bmod 4$, then the shortest cycles are of length $6$, ignoring self-edges. 
\item
If $p \equiv 1 \bmod 3$, then there are $p-3$ cycles of length $6$.
\item
If $p \equiv 2 \bmod 3$, then there are $p+1$ cycles of length $6$. They are of the form
%\begin{center}
%\begin{tikzpicture}[scale=1.5]
%%\draw (0,0) node{\small{$x^2+y^2+1=xy$}};
%\draw (0,0)++(30:1)--++(150:1)--++(210:1)--++(270:1)--++(330:1)--++(30:1)--++(90:1);
%\draw (0,0)++(90:1) node[above]{$(x,y,1)$};
%\draw (0,0)++(150:1) node[left]{$(y-x,y,1)$};
%\draw (0,0)++(30:1) node[right]{$(x,x-y,1)$};
%\draw (0,0)++(210:1) node[left]{$(y-x,-x,1)$};
%\draw (0,0)++(270:1) node[below]{$(-y,-x,1)$};
%\draw (0,0)++(330:1) node[right]{$(-y,x-y,1)$};
%%\foreach \a in {0,1,2}
%%{
%%\draw (120*\a:0.7) node{$m_2$};
%%\draw (120*\a+60:0.7) node{$m_1$};
%%}
%\end{tikzpicture}
%\end{center}
shown in Figure~\ref{fig:square-hex} with $z=1$, or similarly with any of the three coordinates equal to $\pm 1$. 
\end{enumerate}
\end{lemma}

\begin{figure}
\hspace{-4cm}
\begin{tikzpicture}[scale=2]
\draw (0,0)--(1,0)--(1,1)--(0,1)--(0,0);
\draw (0,0) node[below,left]{$(x,y,0)$};
\draw (1,0) node[below,right]{$(-x,y,0)$};
\draw (1,1) node[above,right]{$(-x,-y,0)$};
\draw (0,1) node[above,left]{$(x,-y,0)$};
\draw (0.5,0.5) node{\small{$x^2+y^2=0$}};
\end{tikzpicture}
\\
\hspace{4cm}
\begin{tikzpicture}[scale=1.5]
\draw (0,0) node{\small{$x^2+y^2+1=xy$}};
\draw (0,0)++(30:1)--++(150:1)--++(210:1)--++(270:1)--++(330:1)--++(30:1)--++(90:1);
\draw (0,0)++(90:1) node[above]{$(x,y,1)$};
\draw (0,0)++(150:1) node[left]{$(y-x,y,1)$};
\draw (0,0)++(30:1) node[right]{$(x,x-y,1)$};
\draw (0,0)++(210:1) node[left]{$(y-x,-x,1)$};
\draw (0,0)++(270:1) node[below]{$(-y,-x,1)$};
\draw (0,0)++(330:1) node[right]{$(-y,x-y,1)$};
\end{tikzpicture}
\caption{
Top: the cycles of length 4 from part (3) of Lemma~\ref{lem:counts}, which arise only when $-1$ is a square.
Bottom: a cycle of length 6 from parts (5)-(6) of Lemma~\ref{lem:counts}
}
\label{fig:square-hex}
\end{figure}
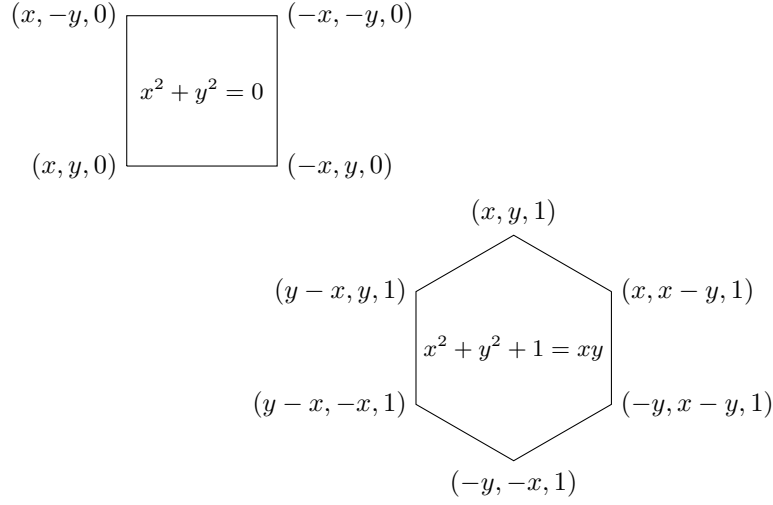

\begin{proof}[Proof of Lemma~\ref{lem:counts}]
Part (1) is due to Carlitz \cite{Carlitz}, and (2) to Cerbu-Gunther-Magee-Peilen \cite[Lemma 2.3]{CGMP}. We review the arguments in Propositions~\ref{prop:level-counts} and \ref{prop:selfies} below, especially to confirm for part (2) that no point has multiple self-edges except $(0,0,0)$.
The enumeration of short cycles is the novel aspect of Lemma~\ref{lem:counts}. It involves two steps: counting the number of squares and hexagons of the form described above, which we do using Proposition~\ref{prop:level-counts} in this Section; and determining whether there are any other short cycles, which we postpone to Sections~\ref{sec:words}, \ref{sec:323121}, and \ref{sec:321321}.
Parts (3) and (4) follow from the enumeration in Section~\ref{sec:words} and Corollary~\ref{prop:2121}.
Propositions~\ref{prop:323121} and \ref{prop:321321} complete the proof of (5) and (6).

The congruences in Lemma~\ref{lem:counts} arise in deciding whether $-1$ and $-3$ have square roots modulo $p$, by quadratic reciprocity. This determines the number of solutions to the Markoff equation with $z=0$ or $z=1$, and hence the number of squares or hexagons of the form above. 
The most subtle case is when $p \equiv 1 \bmod 3$, where the fact that $-3$ is a quadratic residue complicates matters. We must discard solutions of the form $(x,2x,1)$ because the resulting cycles involve self-edges:
\[
x \mapsto (2x)\cdot 1 - x = x
\]
These occur when $x^2 = -1/3$, since the Markoff equation with $y=2x$ and $z=1$ becomes $x^2+4x^2+1=2x^2$. We must discard all six of the triples
\begin{align*}
\left(\frac{1}{\sqrt{-3}},\frac{2}{\sqrt{-3}},1\right), \quad \left(\frac{1}{\sqrt{-3}},-\frac{1}{\sqrt{-3}},1\right), \quad \left(-\frac{2}{\sqrt{-3}},-\frac{1}{\sqrt{-3}},1\right) \\
\left(-\frac{1}{\sqrt{-3}},-\frac{2}{\sqrt{-3}},1\right), \quad \left(-\frac{1}{\sqrt{-3}},\frac{1}{\sqrt{-3}},1\right), \quad \left(\frac{2}{\sqrt{-3}},\frac{1}{\sqrt{-3}},1\right)
\end{align*}
leaving only $p-7$ solutions for $z=1$ instead of the $p-1$ from Proposition~\ref{prop:level-counts}. These form $(p-7)/6$ hexagons, for a total of $p-7$ from all six level sets $x,y,z=\pm 1$. 
To compensate for the loss, there are four additional cycles involving $\sqrt{-3}$, as we describe in Section~\ref{sec:321321}. This gives the final tally $p-3$.
\end{proof}

In part (2), the fixed points of $m_1$, $m_2$, $m_3$ are the self-edges in the Markoff graph mod $p$, which we prefer to delete. One could also think of a self-edge as bounding a face, without changing the Euler characteristic $V-E+F$ since both $E$ and $F$ increase by 1. These can be drawn as small loops avoiding the other edges, so there is no difference for purposes of planarity. 

The number of edges in a connected component with $V$ vertices, after deleting self-edges, satisfies
\begin{equation} \label{eqn:delete-selfies}
E \geq \frac{3}{2} V - \frac{3}{2} \left(p - 4 - (-1)^{(p-1)/2} \right).
\end{equation}
Equality holds if the component contains all the points with self-edges, for instance if the Markoff graph itself is connected. 

To count the number of squares with $z=0$, or hexagons with $z=1$, we use the following proposition going back to Carlitz \cite{Carlitz}. See also \cite[Lemmas 3-4-5]{BGS}. We give a proof for the sake of having all the necessary tools at hand in a common notation.
Throughout, $\legendre{\cdot}{p}$ denotes the Legendre symbol, with value 1 for (non-zero) quadratic residues mod $p$, $-1$ for non-residues, and 0 for 0.
\begin{proposition} \label{prop:level-counts}
Given $z$, the number of solutions $(x,y)$ to
\[
x^2 + y^2 + z^2 = xyz + k
\]
is as follows. If $z^2 \neq 4$ and $z^2 \neq k$, then the number of solutions over $\F_p$ is
\begin{equation} \label{eqn:level-z}
p - \legendre{z^2-4}{p}
\end{equation}
If $z^2 = 4$, then the number is
\begin{equation} \label{eqn:level-2}
\left( 1 + \legendre{k-4}{p} \right) p
\end{equation}
which is either $0$, $p$, or $2p$.
If $z^2 = k \neq 4$, then the number is
\begin{equation} \label{eqn:level-sqrtk}
p + \legendre{k-4}{p} (p-1)
\end{equation}
either $1$ or $2p-1$. 
The total number of solutions $(x,y,z)$ is
\begin{equation} \label{eqn:total-k}
p^2 + \legendre{k-4}{p} \left( 3 + \legendre{k}{p} \right) p + 1 
\end{equation}
This count includes $(0,0,0)$ when $k=0$, which leaves $p^2 \pm 3p$ vertices in the graph.
\end{proposition}

The special cases $z^2=4$ and $z^2=k$ correspond to lines contained in the Markoff cubic surface:
\[
z^2=4 \implies x = \pm y \pm \sqrt{k-4}, \quad
z^2=k \implies x = \left( \sqrt{\frac{k}{4}} + \sqrt{\frac{k-4}{4}} \right)y.
\]
The conic associated to each of the four level sets $z = \pm 2, \pm \sqrt{k}$ is a pair of lines. Setting $x$, $y$, or $z$ equal to any of these levels, we obtain up to 24 of the famous 27 lines on a cubic surface \cite{cayley}. 
The remaining 3 lines on the Markoff surface (over an algebraic extension) are ``at infinity" in projective space. 
Depending on whether $k$ and $k-4$ are quadratic residues, these lines might only become visible in an extension of $\F_p$. Moreover, some of the lines coalesce in the singular cases $k=0$ and $k=4$. 
For $k=0$, a basic difference between $3 \bmod 4$ and $1 \bmod 4$ is that none of the lines are defined over the ground field if $p \equiv 3 \bmod 4$. 
In case $p \equiv 1 \bmod 4$, Figure~\ref{fig:1mod4} shows how to deduce non-planarity from the arrangement of lines in the Markoff surface.

Before turning to the proof of Proposition~\ref{prop:level-counts}, it is worth noting an interpretation of the quadratic symbol in (\ref{eqn:level-z}). The change of variable $z= \zeta + \zeta^{-1}$ has (one-to-two) inverse
\[
\zeta = \frac{ z \pm \sqrt{z^2 - 4} }{2}
\]
so it is the quadratic status of $z^2-4$ that determines whether $\zeta$ lies in $\F_p$ or an extension. 
This change of variable plays a decisive role in the analysis of \cite{BGS}, as we will see in Sections~\ref{sec:alt} and \ref{sec:cage}. Conceptually, if $z$ is the trace of a matrix in ${\rm SL}_2$, then $\zeta$ and $\zeta^{-1}$ are the eigenvalues.

\begin{proof}[Proof of Proposition~\ref{prop:level-counts}]
We fix $z$ and sum the number of solutions $y$ for each $x$. There are 0, 1, or 2 solutions according to the radical that arises in solving the Markoff equation for $y$:
\[
y = \frac{xz \pm \sqrt{x^2 z^2 - 4(x^2+z^2-k)}}{2}
\]
The number of solutions $(x,y)$ is then
\[
\sum_x \left( 1 + \legendre{ x^2(z^2-4) - 4(z^2-k) }{p} \right)
\]
If $z^2=4$, then the summand does not depend on $x$, and one obtains (\ref{eqn:level-2}) since $-4(z^2-k)=4(k-4)$ differs from $k-4$ by a square.
If $z^2=k$, then
\[
\sum_x \left( 1 + \legendre{ x^2(z^2-4) - 4(z^2-k) }{p} \right) = p + \legendre{k-4}{p} \sum_x \legendre{x^2}{p}
\]
and (\ref{eqn:level-sqrtk}) follows since $\legendre{x^2}{p}$ is 0 for $x=0$ and 1 for every other term.

In the remaining cases, the number of solutions is
\[
p + \legendre{z^2-4}{p} \sum_x \legendre{ x^2 - 4(z^2-k)/(z^2-4) }{p}
\]
so (\ref{eqn:level-z}) follows from a convenient fact about quadratic residues: for any non-zero shift $c \neq 0$, 
\begin{equation} \label{eqn:shifted-squares}
\sum_x \legendre{x^2-c}{p} = -1.
\end{equation}
This is a standard fact that can be shown for $c=t^2$ by factoring $x^2-c=(x-t)(x+t)$. The terms $x = \pm t$ contribute 0, and if $x \neq t$, then $(x-t)$ and $(x-t)^{-1}$ are squares or not together. We may then change variable to $u = \frac{x+t}{x-t}$ and obtain a complete character sum missing only $u = 0, 1$ since, in the projective line, $x \neq t, -t, \infty$ corresponds to $u \neq \infty, 0, 1$. This proves (\ref{eqn:shifted-squares}) in case $c$ is a quadratic residue. The sum only depends on whether $c$ is a quadratic residue, so the common value among non-residues can then be obtained by subtraction.
Indeed, changing the order of summation gives $\sum_{c \neq 0} \sum_x \legendre{x^2-c}{p} = \sum_x \sum_{c \neq 0} \legendre{x^2-c}{p}= -(p-1)$, so the value for $c$ not a square must also be $-1$ as in (\ref{eqn:shifted-squares}).

We use this fact once again to sum over $z$ and deduce (\ref{eqn:total-k}). The total is
\[
\sum_{z^2 \neq 4,k} \left( p - \legendre{z^2-4}{p} \right) + \left( 1 + \legendre{k}{p} \right)\left(p + \legendre{k-4}{p}(p-1) \right) + \left( 1 + \legendre{k-4}{p} \right) p
\]
which simplifies as claimed upon collecting the terms in $p^2$, $p$, and $1$. 
\end{proof}

Part (2) of Lemma~\ref{lem:counts} restates the following result of \cite[Lemma 2.3]{CGMP}. We give a proof to highlight a special property of $k=0$ compared to other levels, and to confirm that only $(0,0,0)$ has multiple self-edges.
\begin{proposition}[Cerbu-Gunther-Magee-Peilen, \cite{CGMP}] \label{prop:selfies}
The number of non-zero fixed points of $m_1$ on $x^2+y^2+z^2=xyz$ is
\[
p-4 - \legendre{-1}{p}
\]
and only $(0,0,0)$ is fixed by more than one Markoff move.
\end{proposition}

In particular, the Markoff graph mod 5 has no self-edges. It is drawn in Figure~\ref{fig:5}.

\begin{proof}
The fixed points under $x \mapsto yz-x$ are given by $x = yz/2$. Substituting this into the Markoff equation yields
\[
y^2 \left(1 - \frac{z^2}{4} \right) + z^2 = k.
\]
If $z^2=4$, then necessarily $k=4$. In this case, the fixed points are $(x,x,2)$ and $(x,-x,-2)$ with $x$ arbitrary, and symmetrically $(x,2,x)$ or $(x,-2,-x)$ with the second and third coordinates exchanged.
Assuming $z^2 \neq 4$, we solve for $y$ as:
\begin{equation} \label{eqn:hyperelliptic}
y^2 = 4 \frac{z^2 - k}{z^2 - 4}
\end{equation}
The number of solutions is then a character sum, as before:
\[
\sum_{z^2 \neq 4} \left( 1 + \legendre{ (z^2-k)(z^2-4)^{-1}}{p} \right) = p-2 + \sum_z \legendre{(z^2-k)(z^2-4)}{p} 
\]
where the sum can now be taken over all $z$, with no contribution from $z^2=4$. 
For $k=0$, the factor $z^2-k$ is always a square, with Legendre symbol 0 for $z=0$ or 1 otherwise.
We account for $z=0$ separately, and the sum over all $z$ is given by (\ref{eqn:shifted-squares}) again:
\[
\sum_{z \bmod p} \legendre{(z^2-k)(z^2-4)}{p} = \sum_{z \neq 0} \legendre{z^2-4}{p} = -1 - \legendre{-1}{p}
\]
The total number of fixed points is then $p-3 - \legendre{-1}{p}$, or just $p-4-\legendre{-1}{p}$ excluding $(0,0,0)$. 
For $k \neq 0$ or 4, the number of fixed points is dictated by a curve (\ref{eqn:hyperelliptic}) of genus 1, which in our case degenerates to a conic.

Finally, suppose $(x,y,z)$ is fixed by both $m_1$ and $m_2$. Then $x=yz/2$ and $y=xz/2$, which implies that $y=yz^2/4$, so either $y=0$ or $z^2=4$. If $y=0$, then also $x=yz/2=0$, leaving only $(0,0,\sqrt{k})$, or just $(0,0,0)$ in the case $k=0$. If $z^2 = 4$, which is possible only for $k=4$, then the fixed points are $(x,x,2)$ and $(x,-x,-2)$.
\end{proof}

The proof of Theorem~\ref{thm:euler-char} uses the following result from \cite[Corollary 6.2]{dci-magee}.
\begin{theorem}[de Courcy-Ireland and Magee, \cite{dci-magee}] \label{thm:kesten-mckay}
There is an absolute constant $C > 0$ such that any reduced word of length $L$ in $m_1$, $m_2$, $m_3$ has at most $C^Lp$ fixed points.
\end{theorem}
The constant from \cite{dci-magee} is explicit. For instance, one could take $2^{16L+10}$ in place of $C^L$. 
Recall that a word is reduced if there are no trivial cancellations such as $m_1$ appearing twice in a row.

\section{Euler characteristic and main strategy} \label{sec:euler}

In this section, we prove Lemma~\ref{lem:planar} and Theorem~\ref{thm:euler-char}. We begin the proof of Theorem~\ref{thm:main}, assuming the Markoff graph mod $p$ is connected and using Lemma~\ref{lem:counts}.
We complete the calculations with short cycles in the following sections, and address the possibility of disconnected Markoff graphs in Section~\ref{sec:cage}. 
These arguments prove Theorem~\ref{thm:main} except perhaps for $p=11, 13$. In those cases, the graph is shown to be non-planar using Theorem~\ref{thm:7} for $p=11$ or Theorem~\ref{thm:mod4} for $p=13$.

A planar drawing of a graph divides the plane into connected regions, called faces. Euler's formula states that for a connected graph with $V$ vertices and $E$ edges, dividing the plane into $F$ faces,
\[
V - E + F = 2.
\]
We recommend \cite[Chapter 7]{CBG} as an introduction to Euler characteristic.

To prove Lemma~\ref{lem:planar}, we count the number of pairs $(v,f)$ where a vertex $v$ lies on the boundary of a face $f$. Let us write $v \sim f$ for this incidence relation. By hypothesis, each face has at least $g$ vertices on its boundary. On the other hand, each vertex borders one face for each edge incident to it (with the caveat, for faces incident to a vertex whose removal would disconnect the graph, of counting with multiplicity equal to the number of edges). 
It follows that
\[
gF \leq \sum_f \sum_v \mathbbm{1}[v \sim f] = \sum_v \deg(v) = 2E
\]
since every edge is counted twice, once for each endpoint.
Solving Euler's formula for $F=E-V+2$, we find $gE - g(V-2) \leq 2E$, and therefore
\[
E \leq \frac{g}{g-2} (V-2).
\]
This completes the proof of Lemma~\ref{lem:planar}.

\begin{proof}[Proof of Theorem~\ref{thm:euler-char}]
More generally, for a graph embedded in a surface of Euler characteristic $\chi$, we would have $V-E+F=\chi$.
By the same argument as above,
\[
E \leq \frac{g}{g-2} (V- \chi)
\]
To prove Theorem~\ref{thm:euler-char}, we think of this as a bound for $\chi$ rather than for $E$:
\[
\frac{g-2}{g} E - V \leq - \chi
\]
Since the Markoff graphs mod $p$ do have some short cycles, it is worth introducing some correction terms in order to take a larger value of $g$. Let $n_L$ be the number of faces of length $L$.
Counting pairs $(v,f)$ as above leads to
\[
gF - \sum_{L < g} (g-L) n_L \leq 2E
\]
since most faces are incident to at least $g$ vertices, with a deficit $g-L$ for the shorter faces.
Substituting $F=E-V+\chi$ into this gives
\begin{equation} \label{eqn:euler-inequality}
(g-2)E - gV - \sum_{L<g} (g-L)n_L \leq -g\chi .
\end{equation}

For the Markoff graph mod $p$, or a connected component of it, 
\begin{equation} \label{eqn:edge-lower-bound}
E \geq \frac{3}{2} V - \frac{3}{2} \left( p - 4 - \legendre{-1}{p} \right)
\end{equation}
with equality if the component contains all the self-edges from Lemma~\ref{lem:counts}, part (2).
In all likelihood, the Markoff graph mod $p$ is connected, but in any event our arguments can be applied to a sufficiently large component. This might be of interest for other surfaces $x^2+y^2+z^2=xyz+k$. The component must have at least $p$ vertices in order for (\ref{eqn:edge-lower-bound}) to give a positive number of edges.

With the number of edges $E$ bounded from (\ref{eqn:edge-lower-bound}), inequality (\ref{eqn:euler-inequality}) becomes
\begin{equation} \label{eqn:v-bound}
\frac{1}{2} \left( 1 - \frac{6}{g} \right)V - \frac{1}{2} \left( 3 - \frac{6}{g} \right)\left(p-4-\legendre{-1}{p}\right) - \sum_{L<g} \left( 1 - \frac{L}{g} \right) n_L \leq -\chi
\end{equation}
We will choose $g = \delta \log{p}$ for a sufficiently small $\delta$, or rather a nearby integer $\lfloor \delta \log{p} \rfloor$.
This ensures that there are few short faces, by Theorem~\ref{thm:kesten-mckay}. Each face is outlined by a word in the Markoff moves $m_1$, $m_2$, $m_3$, up to cyclic ordering, and identifying a word and its inverse as the two orientations of the face. The word, or one of its cyclic shifts, fixes the vertices on the boundary of the face. 
It follows from Theorem~\ref{thm:kesten-mckay} that
\[
\sum_{L<g} \left(1 - \frac{L}{g} \right) n_L \leq \sum_{L<g} C^L p \lesssim p^{1+\varepsilon}
\]
for any desired $\varepsilon > 0$, if $\delta$ is chosen small enough. We write $\lesssim$ for inequality up to a constant multiple, independent of $p$, but perhaps depending on $\varepsilon$. 
Note that the value of $C$ is not the same as before: we multiply by the number of reduced words of length $L$, which is roughly $2^L$, effectively enlarging the previous bound $2^{16L+10}$ to $2^{17L+10}$.

For the giant component of Bourgain-Gamburd-Sarnak, which contains almost all the vertices, we have $V \sim p^2$. The remaining terms are negligible in comparison:
\[
\frac{1}{2} \left( 1 - O\left( \frac{1}{\log{p} } \right) \right) p^2 + O\left( p^{1+\varepsilon} \right) \leq - \chi
\]
Theorem~\ref{thm:euler-char} follows.
\end{proof}

We can now prove Theorem~\ref{thm:main}, assuming the Markoff graph mod $p$ is connected.
It simplifies the calculations if the graph is connected, but in any case, Section~\ref{sec:cage} gives an unconditional proof showing that the giant component is not planar.
For the proof of Theorem~\ref{thm:main}, we take $g=7$ rather than $g \approx \log{p}$, and estimate $n_L$ directly for all $L \leq 6$. 
Let $s$ be the number of square faces and $h$ the number of hexagons -- for any drawing of the graph, these are at most the number of 4-cycles or 6-cycles.
As we will see in Section~\ref{sec:words}, there are no triangles or pentagons. There are then $g$ vertices per face, with a deficit of 1 for each hexagon, and 3 for each square. The key inequality (\ref{eqn:v-bound}) becomes
\[
\frac{1}{2} \left( 1 - \frac{6}{7} \right)V - \frac{1}{2} \left( 3 - \frac{6}{7} \right)\left(p-4-\legendre{-1}{p}\right) - \frac{h+3s}{7} \leq -\chi
\]
We think of this as an upper bound for $V$, which cannot hold once $p$ is large enough. Recall that $\chi=2$ for the plane:
\begin{equation} \label{eqn:p-bound}
V \leq 15 \left(p-4-\legendre{-1}{p}\right) + 2h + 6s -28
\end{equation}

In view of Lemma~\ref{lem:counts}, we consider four cases depending on $p$ modulo 3 and 4. The cases are $p \equiv 1, 5, 7, 11 \bmod 12$. 

If $p \equiv 1 \bmod 12$, then (at the most, for any drawing) the number of squares is $s=3(p-1)/2$ and the number of hexagons is $h=p-3$. The inequality (\ref{eqn:p-bound}) becomes
\begin{equation} \label{eqn:1mod12}
V \leq 26p - 118
\end{equation}
Assuming the Markoff graph mod $p$ is connected, we can take $V=p^2+3p$ and solve the quadratic inequality $p^2+3p \leq 26p-118$. For $p=13$, it does hold in the form $208 \leq 220$, so this case warrants a separate argument. The next example of this form is $p=37$, and already non-planarity follows because
\[
p^2+3p=1480 > 844 = 26p-118.
\]

If $p \equiv 5 \bmod 12$, the number of squares is $s=3(p-1)/2$, and the number of hexagons is $h=p+1$. 
The inequality (\ref{eqn:p-bound}) becomes
\begin{equation} \label{eqn:5mod12}
V \leq 26p - 110
\end{equation}
If one knew $V=p^2+3p$, planarity would be possible only for
\[
6.78\ldots = \frac{23-\sqrt{89}}{2} \leq p \leq \frac{23+\sqrt{89}}{2} = 16.21\ldots
\]
Even 5 and 17, the smallest primes of this form, therefore have non-planar Markoff graphs.
%d=23^2-4*110
% 89
% (23+sqrt(89))/2
% 16.216990566028301905660330188811320359
%(23-sqrt(89))/2
% 6.7830094339716980943396698111886796415

If $p \equiv 7 \bmod 12$, there are no squares and the number of hexagons is $p-3$. 
The inequality (\ref{eqn:p-bound}) becomes
\begin{equation} \label{eqn:7mod12}
V \leq 17p - 79
\end{equation}
Assuming the Markoff graph mod $p$ is connected, we can take $V=p^2 - 3p$ since $p \equiv 3 \bmod 4$ in this case. The inequality is already impossible for $p=19$, the first candidate after $p=7$ in this progression mod 12. Indeed, the larger root of $p^2-3p = 17p-79$ is
\[
10 + \sqrt{21} \approx 14.58\ldots
\]
%14.582575694955840006588047193728008489

If $p \equiv 11 \bmod 12$, there are no squares and the number of hexagons is $p+1$.
The inequality (\ref{eqn:p-bound}) becomes
\begin{equation} \label{eqn:11mod12}
V \leq 17p - 71
\end{equation}
Assuming connectedness, this inequality shows non-planarity for 
\[
p > 10+\sqrt{29} = 15.38\ldots
\]
The first prime $p=11$ in this progression requires special treatment. For example, modulo 11, we have $-7 = 4 = 2^2$, so non-planarity follows from Theorem~\ref{thm:7}.
The other case left after the arguments above is $p=13$, which has a non-planar Markoff graph by Theorem~\ref{thm:mod4}.

%10+sqrt(29)
% 15.385164807134504031250710491540329556

\section{Short words} \label{sec:words}

In this section, we advance the proof of Lemma~\ref{lem:counts} by identifying which words in $m_1$, $m_2$, $m_3$ can possibly bound a face of 6 sides or less. 
Up to a permutation of the coordinates, we may assume the word's first move is $m_1$, followed by $m_2$. We need only consider reduced words, where none of the involutions $m_1$, $m_2$, or $m_3$ occurs twice in a row.
It is convenient to omit the $m$'s, simply writing $j$ for $m_j$.
The words of length up to 6 are then
\begin{align*}
1 \\
21 \\
121, 321 \\
2121, 3121, 1321, 2321 \\
12121, 32121, 13121, 23121, 21321, 31321, 12321, 32321 \\
212121, 232121, 313121, 323121, 321321, 231321, 312321, 132321 \\
312121, 132121, 213121, 123121, 121321, 131321, 212321, 232321
\end{align*}

We can immediately discard words where some move occurs only once, such as 232321. These do not give new faces in the Markoff graph, but simply add a self-edge somewhere along a face that has already been counted. 

Likewise, there is no contribution from words that have a shorter conjugate. The fixed points of $w^{-1}w_0w$ do not yield new faces in the Markoff graph. Instead, one applies $w$ to the fixed point, traverses a face bounded by $w_0$, and returns along the same path. 

After deleting words with a lone letter or a shorter conjugate, we are left with
\[
2121, \quad 212121, \quad 321321,
\quad 323121, \quad 231321, \quad 312321
\]
The last three are equivalent to each other under cyclic shifts and permutations of the coordinates:
\begin{align*}
312321 \sim 131232 \\
231321 \sim 123132 \sim 212313
\end{align*}
where the rightmost words have the same structure as 323121 up to a permutation. We will see in Section~\ref{sec:323121} that these words do not bound any faces. The remaining cases 2121, 212121, and 321321 will be treated in Sections~\ref{sec:321321} and \ref{sec:alt}, completing the proof of Lemma~\ref{lem:counts}.

\section{Fixed points of 323121} \label{sec:323121}

\begin{proposition} \label{prop:323121}
The fixed points of $m_3m_2m_3m_1m_2m_1$ on the Markoff surface $x^2+y^2+z^2=xyz$ are the triples $(x,y,z)$ satisfying
\[
x^4-5x^2 +8 = 0, \quad z = \pm x, \quad y = \frac{xz}{x^2-2}
\]
together with $(0,0,0)$.
These do not correspond to faces in the Markoff graph mod $p$. Instead, there are self-edges $m_2$ at both neighbours of the fixed point $(x,y,z)$ under $m_1$ and $m_3$.
\end{proposition}

For example, this occurs in the Markoff graph mod 11 with $x=3$ (Figure~\ref{fig:11}). The self-edges correspond to $6 = 3 \times 4 - 6$ at $(3,6,4)$ or $(4,6,3)$. 
Over other fields, one needs to have an element
\[
x = \sqrt{ \frac{5 + \sqrt{-7}}{2} }
\]

\begin{proof}
To lower the degree of the fixed point system, note that $m_{323121}$ fixes $(x,y,z)$ if and only if
\[
m_3 m_2 m_3(x,y,z) = m_1 m_2 m_1 (x,y,z)
\]
\[
 \begin{pmatrix} x \\ x(xy-z)-y \\ x(x(xy-z)-y)-xy+z \end{pmatrix} = \begin{pmatrix} z(z(yz-x)-y)-yz+x \\ z(yz-x)-y \\  z \end{pmatrix}
\]
This simplifies to
\[
\begin{cases} z(z(yz-x)-2y) &=0 \\ y(z^2-x^2) &=0 \\ x(x(xy-z)-2y) &=0 \end{cases}
\]
Assuming $xyz \neq 0$, we find that $z^2 = x^2$ and solve for $y$ from $(x^2-2)y=xz$.
Substituting this into the Markoff equation $x^2+y^2+z^2=xyz+k$, we are left with a single-variable sextic for $x$:
\[
\frac{(x^4-5x^2 + 8)x^2}{(x^2-2)^2} = k
\]
For $k=0$, this reduces to a biquadratic equation $x^4-5x^2 + 8=0$ as claimed, assuming $x \neq 0$.
If $x = 0$, and likewise if $y$ or $z$ vanishes, then the system implies that at least two variables must vanish. The only such solutions of the Markoff equation are $(0,0,0)$ for $k \neq 0$, or more generally the permutations of $(0,0,\pm \sqrt{k})$ for other levels.

From $z(yz-x)-2y=0$, we see that $m_2$ fixes $(yz-x,y,z)$, and similarly for $(x,y,xy-z)$. This shows that there are self-edges at the neighbours of $(x,y,z)$, as claimed and completing the proof.
\end{proof}

\section{Fixed points of 321321} \label{sec:321321}

The situation here depends on whether $-3$ is a quadratic residue modulo $p$. If so, then the next proposition shows that there are four hexagons fixed by 321321 and its cyclic shifts. 
Permutations of the coordinates do not lead to any further hexagons: the permuted words are either cyclic shifts 132132 and 213213, or their inverses, which all bound the same faces.
This case accounts for the four hexagons visible in the Markoff graph mod 7 (Figure~\ref{fig:mod7}).
\begin{proposition} \label{prop:321321}
The fixed points of $m_3m_2m_1 m_3m_2m_1$ on $x^2+y^2+z^2=xyz$ are the triples $(x,y,z)$ satisfying
\[
y^2+3y+3 = 0, \quad x^2=\frac{y^2}{(y+1)^2}, \quad z=-x
\]
or
\[
y^2-3y+3 = 0, \quad x^2 = \frac{y^2}{(y-1)^2}, \quad z=x
\]
\end{proposition}

\begin{proof}
The fixed points are given by $m_{321321}(x,y,z)=(x,y,z)$, or equivalently
\[
m_1 m_2 m_3 (x,y,z) = m_3 m_2 m_1(x,y,z)
\]
\[
\begin{pmatrix} (xy-z)(x(xy-z)-y) - x \\ x(xy-z) - y \\ xy-z \end{pmatrix} = \begin{pmatrix} yz-x \\ z(yz-x)-y \\ (yz-x)(z(yz-x)-y)-z \end{pmatrix}
\]
This simplifies to
\[
\begin{cases} x( (xy-z)^2 - y^2) &= 0 \\ y(x^2-z^2) &=0 \\ z((yz-x)^2 - y^2) &= 0 \end{cases}
\]
Suppose that $xyz \neq 0$. Then $x^2 = z^2$ from the middle equation, and this leads to a redundancy. Since $z = \pm x$, we have $xy-z = \pm (yz-x)$ and the remaining two equations become equivalent. We consider the two cases $z=\pm x$ separately and solve for $x$ from $y^2=(xy-z)^2 = x^2(y \mp 1)^2$. 
Substituting this relation and $z = \pm x$ into the Markoff equation, one finds
\[
x^2+y^2+z^2=xyz+k \implies 2 \frac{y^2}{(y \mp 1)^2} + y^2 = \pm \frac{y^3}{(y \mp 1)^2} + k
\]
For $k=0$, assuming $y\neq 0$, we divide by $y^2$ and obtain the two quadratics from the statement of the Proposition.

It remains to consider the possibility that some of $x$, $y$, $z$ could be 0. If $x=0$, then $y(x^2-z^2)=0$ implies that either $y$ or $z$ must also be 0. This leaves only $(0,0,0)$ as a fixed point on the original surface $x^2+y^2+z^2=xyz$, or more generally permutations of $(0,0,\sqrt{k})$ on $x^2+y^2+z^2=xyz+k$.
\end{proof}

The solutions for $y$ are $\frac{1}{2}(\pm 3\pm \sqrt{-3})$. Modulo 7, we choose $\sqrt{-3} = \pm 2$ and obtain for example $(x,y,z)=(4,1,3)$ from one of the hexagons of Figure~\ref{fig:mod7}. 

\section{Alternating words} \label{sec:alt}

The remaining words 2121 and 212121 do not change the third coordinate $z$, and act linearly on $(x,y)$. 
By diagonalizing this action, one can determine the fixed points of any alternating word $(21)^L$.
\begin{proposition} \label{prop:alt}
The only fixed points of $(m_2 \circ m_1)^L$ on $x^2+y^2+z^2=xyz+k$ are $(0,0,\sqrt{k})$ unless $L$ is divisible by $p$, or $L$ is a factor of $(p-1)/2$ or $(p+1)/2$.
\begin{enumerate}
\item[(a)]
If $L$ is divisible by $p$, then the fixed points are $(x,y,\pm 2)$ together with $(0,0,\sqrt{k})$.
The former lie on the lines $(x \mp y)^2 = k-4$. 
\item[(b)]
If $(p \pm 1)/2$ is divisible by $L$, then the fixed points of $(m_2 \circ m_1)^L$ are $(x,y,z)$ where $z = \zeta + \zeta^{-1}$ with $\zeta \in \F_{p^2}^{\times}$ a solution of
\[
\zeta^{2L} = 1, \quad \zeta \neq \pm 1
\]
together with $(0,0,\sqrt{k})$ for the level $x^2+y^2+z^2=xyz+k$. 
\end{enumerate}
\end{proposition}

The form of the fixed points in (a) and (b) does not depend on $k$. One simply imposes $x^2+y^2+z^2=xyz+k$ in addition to the fixed-point system, and includes also the exceptional points $(0,0,\sqrt{k})$ where two coordinates equal $0$.

The case 2121, where $L=2$, corresponds to $z=0$ and $\zeta=\sqrt{-1}$. However, if $k=0$, even though we allow $\zeta$ in a quadratic extension, the value $z=0$ is only possible for $p \equiv 1 \bmod 4$.
Substituting $z=0$ in the Markoff equation $x^2+y^2+z^2=xyz$ gives $x^2 + y^2 = 0$. If $p \equiv 3 \bmod 4$, then the only solution is $(0,0,0)$, or else $(x/y)^2 = -1$. 
\begin{corollary} \label{prop:2121}
The fixed points of 2121 are $(x,y,0)$ with
\[
x^2 + y^2 = 0
\]
If $p \equiv 3 \bmod 4$, the only fixed point on $x^2+y^2+z^2=xyz \bmod p$ is $(0,0,0)$. 
\end{corollary}

The case 212121, where $L=3$, corresponds to $z = \pm 1$ with
\[
\pm \zeta = \frac{1 + \sqrt{-3}}{2}
\]
\begin{corollary} \label{prop:212121}
The fixed points of 212121 are $(x,y, \pm 1)$ with
\[
x^2 + y^2 +1 = \pm xy
\]
\end{corollary}

\begin{proof}[Proof of Proposition~\ref{prop:alt}]
The action of $m_2 \circ m_1$ is
\[
\begin{pmatrix} x \\ y \\ z \end{pmatrix} \overset{m_1}{\longrightarrow} \begin{pmatrix} yz-x \\ y \\ z \end{pmatrix} \overset{m_2}{\longrightarrow} \begin{pmatrix} yz-x \\ z(yz-x) - y \\ z \end{pmatrix}
\]
that is,
\[
\begin{pmatrix} x \\ y \end{pmatrix} \mapsto \begin{pmatrix} -1 & z \\ -z & z^2-1 \end{pmatrix} \begin{pmatrix} x \\ y \end{pmatrix}
\]
This matrix has determinant 1 and trace $z^2-2$. In terms of a change of variable
\[
z = \zeta + \zeta^{-1}
\]
the eigenvalues are then $\zeta^2$ and $\zeta^{-2}$. 
Here, $\zeta$ may lie in a quadratic extension of $\F_p$ if need be. In order to have $z = \zeta + \zeta^{-1}$ belong to $\F_p$, it must be that either $\zeta^{p+1} = 1$ or $\zeta^{p-1}=1$. 
The order of $\zeta^2$ in $\F_{p^2}^{\times}$, is therefore a divisor of $(p-1)/2$ or $(p+1)/2$. 

For $z \neq \pm 2$, the matrix representing $m_2 m_1$ can be diagonalized, and its order is the order of $\zeta^2$. After computing the eigenvectors, we find
\[
\begin{pmatrix} -1 & z \\ -z & z^2-1 \end{pmatrix} = \begin{pmatrix} 1 & \zeta \\ \zeta & 1 \end{pmatrix} \begin{pmatrix} \zeta^2 & 0 \\ 0 & \zeta^{-2} \end{pmatrix} \begin{pmatrix} 1 & \zeta \\ \zeta & 1 \end{pmatrix}^{-1}
\]
and
\[
\begin{pmatrix} -1 & z \\ -z & z^2-1 \end{pmatrix}^L = \begin{pmatrix} 1 & \zeta \\ \zeta & 1 \end{pmatrix} \begin{pmatrix} \zeta^{2L} & 0 \\ 0 & \zeta^{-2L} \end{pmatrix} \begin{pmatrix} 1 & \zeta \\ \zeta & 1 \end{pmatrix}^{-1}
\]
If $\zeta^{2L}=1$, then every vector $(x,y)$ is fixed, as claimed. Conversely, if $\zeta^{2L} \neq 1$, only $(0,0)$ is fixed. This gives only $(0,0,0)$ in the Markoff surface, or more generally $(0,0,\sqrt{k})$ for other level sets $x^2+y^2+z^2=xyz+k$. 

 If $z = \pm 2$, then $\zeta = \pm 1$ so there is a repeated eigenvalue $\zeta^2 = \zeta^{-2} = 1$, and the eigenvectors above become multiples of each other by $\pm 1$. In this case, $m_2 m_1$ has order $p$ in view of the following Jordan form:
\[
\begin{pmatrix} -1 & z \\ -z & z^2-1 \end{pmatrix} = \begin{pmatrix} -1 & \pm 2 \\ \mp 2 & 3 \end{pmatrix} = \begin{pmatrix} \pm 2 & 0 \\ 2 & 1 \end{pmatrix} \begin{pmatrix} 1 & 1 \\ 0 & 1 \end{pmatrix} \begin{pmatrix} \pm 2 & 0 \\ 2 & 1 \end{pmatrix}^{-1}
\]
The powers of $m_2 m_1$ are given by
\[
\begin{pmatrix} -1 & z \\ -z & z^2-1 \end{pmatrix}^L = \begin{pmatrix} -1 & \pm 2 \\ \mp 2 & 3 \end{pmatrix} = \begin{pmatrix} \pm 2 & 0 \\ 2 & 1 \end{pmatrix} \begin{pmatrix} 1 & L \\ 0 & 1 \end{pmatrix} \begin{pmatrix} \pm 2 & 0 \\ 2 & 1 \end{pmatrix}^{-1}
\]
If $L$ is divisible by $p$, then every vector $(x,y)$ is fixed. 
If $L$ is not divisible by $p$, then the fixed points are given by $x = \pm y$ with the same sign as in $z=\pm 2$. Substituting this into $x^2+y^2+z^2=xyz+k$ gives $2x^2 + 4 = 2x^2+k$. There are no such fixed points, unless $k=4$. 
\end{proof}

\section{Non-planarity of the cage} \label{sec:cage}

In this section, we show that even if the Markoff graph mod $p$ is disconnected, it has a large non-planar component. This is the giant component constructed by Bourgain-Gamburd-Sarnak from what they call \emph{the cage} \cite[Section 3.2]{BGS}. The cage consists of triples $(x,y,z)$ where at least one of the coordinates has maximal order with respect to the analysis from Section~\ref{sec:alt}. It is shown in \cite{BGS} that all of these points belong to the same connected component.

Recall the change of variable
\[
z = \zeta + \zeta^{-1}, \quad \zeta^{p+1}=1 \ \text{or} \ \zeta^{p-1}=1
\]
The maximal order is therefore $p+1$. The number of elements of order $p+1$ in the cyclic group $\F_{p^2}^{\times}$ is given by Euler's totient function $\phi(p+1)$. 
These correspond to $\frac{1}{2} \phi(p+1)$ values of $z=\zeta+\zeta^{-1}$. We ignore the possibility that $\zeta = \zeta^{-1}$, since then $z = \pm 2$. This arises only for $p \equiv 1 \bmod 4$, in which case we might as well conclude non-planarity from Theorem~\ref{thm:mod4}.
The configuration used to prove Theorem~\ref{thm:mod4} meets every level set where a coordinate $x$, $y$, or $z$ takes a given value, as will be clear from (\ref{eqn:inductively}), and in particular it lies in the same component as the cage.

For each of these maximal values of $z$, there are $p+1$ solutions $(x,y)$, by equation (\ref{eqn:level-z}). Indeed, in these cases, $\zeta^{p}=\zeta^{-1}$ so $\zeta$ is ``imaginary" and there are $p+1$ solutions $(x,y)$ rather than $p-1$. There are then $\frac{1}{2} p \phi(p+1)$ triples $(x,y,z)$ where $z$ has maximal order, and similarly for the first or second coordinate. Of course, more than one coordinate could have maximal order.

An interesting example is $p=7$, where the cage encloses the entire graph. In this case, $p+1=8$ so let $\zeta$ be an eighth root of unity. Write $i^2=-1$ in the quadratic extension of $\F_7$, and observe that $3^2 = 9 \equiv 2 \bmod 7$. The maximal order therefore occurs for $z$ equal to
\[
\frac{1+i}{\sqrt{2}} + \frac{1-i}{\sqrt{2}} = \frac{2}{\sqrt{2}} = \sqrt{2} = \pm 3
\]
The Markoff graph mod 7 (Figure~\ref{fig:mod7}) has four vertices such as $(3,3,3)$ up to sign changes, where all coordinates have maximal order. The twelve neighbours of those, such as $(6,3,3)$, have two coordinates of maximal order. Another twelve points, such as $(1,6,3)$, have only one maximal coordinate. These account for all solutions in the form $28=4+12+12 = p\phi(p+1)$. There is a cycle of length 8 at every vertex, with three such octagons meeting at $(3,3,3)$; two octagons and a hexagon at $(6,3,3)$; or an octagon, a hexagon, and a self-edge at $(1,6,3)$.

The points in the cage show that there is a connected component of size at least 
\begin{equation} \label{eqn:cage-lower-bound}
V \geq \frac{1}{2} p\phi(p+1) > \frac{p^2}{1000 \log \log{p}}
\end{equation}
where we have used a loose estimate for Euler's totient function $\phi$. 
Asymptotically, a formula of Mertens gives
\[
\phi(n) \geq \left( e^{-\gamma} + o(1) \right) \frac{n}{\log{\log{n}}}
\]
where $\gamma$ is the Euler-Mascheroni constant and $e^{-\gamma} \approx 0.5614$. See \cite[Theorem 7]{I} or \cite[Theorem 429]{HW}. The correct constant is much larger than the underestimate $1/500$ from (\ref{eqn:cage-lower-bound}), but perhaps only applicable for large $n$. The rougher form (\ref{eqn:cage-lower-bound}) is valid for all $p$ and follows from Chebyshev-style estimates for prime numbers.
We discuss these in Section~\ref{sec:totient}.
%gp > exp(-Euler())
% 0.56145948356688516982414321479088078677

We substitute (\ref{eqn:cage-lower-bound}) in (\ref{eqn:p-bound}), where the number of squares is $s=0$ since we are now interested only in $p \equiv 3 \bmod 4$. The number of hexagons is at most $p+1$, as in (\ref{eqn:11mod12}).
If the connected component of the cage is planar, it follows that
\begin{equation} \label{eqn:tests}
\frac{1}{2}p \phi(p+1) \leq 17p - 71, \quad \phi(p+1) < 34.
\end{equation}
Even with a crude bound for $\phi$, this implies
\begin{equation} \label{eqn:8500}
\frac{p+1}{\log\log(p+1)} \leq 500 \phi(p+1) < 17000
\end{equation}
The solution to $x/\log\log{x}=17000$, using Newton's method for instance, is $x=40134.5\ldots$. In particular, the inequality in (\ref{eqn:8500}) is reversed if $p > 40133$ (which factors as $67 \times 599$, the nearest prime being 40129). 
One could certainly narrow the search further using better estimates, but it is already feasible to compute $\phi(p+1)$ for all primes up to $40129$ (and we only need those congruent to 3 mod 4). 
The criterion (\ref{eqn:tests}) is satisfied for $p \leq 101$, but no larger primes. 
The congruences modulo 28 from Corollary~\ref{cor:28} show non-planarity for several of these, leaving only
\[
p= 7, 19, 31, 47, 59, 83.
\]
Of the remaining cases, $p=7$ does in fact have a planar Markoff graph, which coincides with the cage.
The others are small enough that one can check the graph is connected, either by enumerating enough triples $(x,y,z)$, or by a spectral method (Section ~\ref{sec:lt}; see also \cite{dci-lee} for connectedness up to $p \leq 2999$). The true value $V$ is then even larger than the lower bound from the cage, and non-planarity follows from the reasoning in Section~\ref{sec:euler}.

Figure~\ref{fig:19} shows part of the Markoff graph for $p=19$, the smallest example where neither Theorem~\ref{thm:mod4} nor Theorem~\ref{thm:7} nor the lower bound from the cage is enough to deduce non-planarity.

\section{Lower bound for Euler's totient function} \label{sec:totient}

In this section, we prove the estimate (\ref{eqn:cage-lower-bound}) for Euler's function $\phi(n)$, given by
\begin{equation} \label{eqn:phi-product}
\phi(n) = n \prod_{p \mid n} \left( 1 - \frac{1}{p} \right) \geq \frac{1}{500} \frac{n}{\log\log{n}}
\end{equation}
This is a standard topic, with excellent expositions available in \cite[Theorem 7]{I}, \cite[Theorem 429]{HW}, and \cite[Section 2.2]{FI}. We follow them closely, and simply keep track of the implicit constants.

Consider the contributions to (\ref{eqn:phi-product}) from large primes $p > L$ and small primes $p \leq L$. Eventually, a good choice will be $L = \log{n}$. There are not too many large factors of $n$, because $n > L^k$ if there are $k$ large primes among the factors of $n$. 
Therefore $k < \log{n}/\log{L}$, and each $p > L$ contributes at least $1-1/L$ to the product.
For the small primes, we obtain a lower bound by extending the product to all $p \leq L$, regardless of whether they divide $n$, since each term $1-1/p$ is less than 1.
It follows that
\[
\prod_{p \mid n } \left(1 - \frac{1}{p} \right) > \left( 1 - \frac{1}{L} \right)^{\log{n} / \log{L} } \prod_{p \leq L} \left(1 - \frac{1}{p} \right)
\]
The exponent $\log{n}/\log{L}$ is greater than 1, so the binomial expansion gives
\[
\left( 1 - \frac{1}{L} \right)^{\log{n}/\log{L}} \geq 1 - \frac{\log{n}}{L \log{L}}
\]
which will be bounded below if one chooses $L \asymp \log{n}$ or larger. With $L = C \log{n}$, the contribution of large primes is at least
\begin{equation} \label{eqn:large-half}
1 - \frac{1}{C \log{L}}  \geq \frac{1}{2}, \quad \text{for} \ n \geq e^{e^{2/C}}.
\end{equation}

The decisive contribution, that of small primes, is given asymptotically by a formula of Mertens:
\[
\prod_{p \leq L} \left( 1 - \frac{1}{p} \right) \sim \frac{e^{0.5772\ldots} }{\log{L} }
\]
where the value in the exponent is the Euler-Mascheroni constant \cite[Theorem 7]{I}. For our purposes, it is better to have a less precise estimate that applies already for small values of $L$. 
%0.5772156649015328606065120901

We first take logarithms to convert the product to a sum, and then extract the leading term from the power series $\log(1-x)=-x + \ldots$ obtaining:
\begin{equation} \label{eqn:log}
\prod_{p \leq L} \left(1 - \frac{1}{p} \right) = \exp\left( - \sum_{p \leq L} \frac{1}{p} + \sum_{p \leq L} \left(\frac{1}{p} + \log\left(1 - \frac{1}{p}\right) \right)  \right)
\end{equation}
The second sum converges, since its terms are dominated by $p^{-2}$. Numerically
\[
\sum_{p \leq L}  \left(\frac{1}{p} + \log\left(1 - \frac{1}{p}\right) \right) \geq \sum_p  \left(\frac{1}{p} + \log\left(1 - \frac{1}{p}\right) \right) = -0.3157\ldots
\]
%gp > S=0;
% forprime(p=2,10^4,S+=1/p+log(1-1/p))
%S
%  -0.31571354403274130350848778316649935731
%) gp > S=0;
%forprime(p=2,10^10,S+=1/p+log(1-1/p))
% S
% -0.31571845205180564941447137380837623602

The main term in (\ref{eqn:log}) is therefore $\sum_{p\leq L} 1/p$, which is well known to be of order $\log\log{L}$ (as discussed in the same reference \cite[Theorem 7]{I} for instance). 
For an explicit bound of this form, we first sum by parts:
\begin{equation} \label{eqn:parts}
\sum_{p \leq L} \frac{1}{p} = \sum_{p \leq L} \frac{\log{p}}{p} \frac{1}{\log{p}} = \int_2^L \frac{S(t)}{t (\log{t})^2} dt + \frac{S(L)}{\log{L}}
\end{equation}
where we have differentiated $1/\log{p}$ and integrated $\log{p}/p$. 
The summatory function can be bounded by extending the range to include prime powers:
\begin{equation} \label{eqn:summatory}
S(t) = \sum_{p \leq t} \frac{\log{p}}{p} \leq \sum_{m \leq t} \frac{\Lambda(m)}{m}
\end{equation}
where $\Lambda(m)=\log{p}$ if $m$ is a power of a prime $p$, and 0 otherwise.
These weights are more convenient because of the identity
\[
\sum_{d | n} \Lambda(d) = \log{n}
\]
and its sum
\[
\sum_{d \leq L} \Lambda(d) \left\lfloor \frac{L}{d} \right\rfloor = \sum_{\ell \leq L } \log{\ell}
\]
We multiply and divide by $L$, noting that $\lfloor L/d \rfloor \leq L/d+1$
\[
\sum_{d \leq L} \frac{\Lambda(d)}{d} \leq \frac{1}{L} \sum_{\ell \leq L} \log{\ell} + \frac{1}{L} \sum_{d \leq L} \Lambda(d)
\]
The first sum can be estimated by an integral:
\begin{equation} \label{eqn:integral-test}
\int_1^x \log{t} \ dt \leq \sum_{\ell \leq x} \log{\ell} \leq \int_1^{x+1} \log{t} \ dt = (x+1) \log(x+1) - x
\end{equation}
For the remainder, we claim that the following Chebyshev-style estimate holds already for any $x \geq 2$
\begin{equation} \label{eqn:chebyshev}
\sum_{d \leq x} \Lambda(d) \leq x \log{4} +  \big(\log{x} + 2)\frac{\log{x}}{\log{2}}
\end{equation}
Assuming this for the moment, we continue with (\ref{eqn:summatory})
\begin{align*}
S(t) &\leq \sum_{m \leq t} \frac{\Lambda(m)}{m} \leq \frac{1}{t} \sum_{\ell \leq t} \log{\ell} + \frac{1}{t} \sum_{d \leq t} \Lambda(d) \\
&\leq \log{t} -1+ \log{4} + \frac{(\log{t}+2)\log{t}}{t \log{2}} \leq \log{t} + 2
\end{align*}
Finally, we substitute this into (\ref{eqn:parts}) and find
\[
\sum_{p \leq L} \frac{1}{p} \leq \int_2^L \frac{\log{t}+2}{t(\log{t})^2} dt + \frac{\log{L}+2}{\log{L}}
\]
The integral can be computed exactly by a substitution $u=\log{t}$ with $du = dt/t$, whence
\begin{align*}
\sum_{p \leq L} \frac{1}{p} &\leq \log{\log{L}} - \log\log{2} + 2\left( \frac{1}{\log{2}} - \frac{1}{\log{L}} \right) + \frac{\log{L}+2}{\log{L}} \\
&\leq \log{\log{L}} + 5
\end{align*}
The original product from (\ref{eqn:log}) is then, with $L = \log{n}$,
\begin{equation} \label{eqn:final-product}
\prod_{p \leq L} \left( 1 - \frac{1}{p} \right) \geq \exp(-\log{\log{L}} - 5 - 0.3157) \geq \frac{1}{250 \log{\log{n}}}
\end{equation}
The loose estimate (\ref{eqn:cage-lower-bound}) gives up an extra factor of 2 from the large primes. This is guaranteed by (\ref{eqn:large-half}) for $n \geq e^{e^2} \approx 1618$, and one can check the smaller values of $n$ to be sure (with room to spare) that $\phi(n) > \frac{1}{500} n / \log\log{n}$ for all $n \geq 2$.

To prove (\ref{eqn:chebyshev}), recall the notation $\psi(x) = \sum_{n \leq x} \Lambda(n)$. In terms of $\psi$,
\[
\sum_{m \leq x} \psi(x/m) = \sum_{d \leq x} \Lambda(d) \left\lfloor \frac{x}{d} \right\rfloor = \sum_{n \leq x} \log{n}
\]
Subtraction gives
\begin{align*}
\sum_{n \leq x} \log{n} - 2 \sum_{n \leq x/2} \log{n} &= \sum_{m \leq x} \psi\left( \frac{x}{m} \right) - 2 \sum_{m \leq x/2} \psi\left( \frac{x}{2m} \right) \\
&\geq \psi(x) - \psi(x/2)
\end{align*}
because each difference $\psi(x/(2j-1)) - \psi(x/(2j))$ is non-negative.
This can be simplified using (\ref{eqn:integral-test}) for the logarithms:
\begin{align*}
\psi(x) &\leq \psi(x/2) + x\log{2} + \log{x} + (x+1)\log\left(1+\frac{1}{x}\right) \\
&< \psi(x/2) + x\log{2} + \log{x} + 2
\end{align*}
This can be iterated to bound $\psi(x)$ in terms of $\psi(x/2)$, then $\psi(x/4)$, $\psi(x/8)$, and so on.
After roughly $k \sim \log{x} / \log{2}$ iterations, we reach an empty sum $\psi(x/2^k) = 0$, leaving only a geometric progression:
\begin{align*}
\psi(x) &< x  \left(1 + \frac{1}{2} + \ldots \right)\log{2} + \big(\log{x} + 2)\frac{\log{x}}{\log{2}} \\
&< x \log{4} + \big(\log{x} + 2)\frac{\log{x}}{\log{2}}
\end{align*}
as required.

For comparison, although it is only the upper bound that is relevant in our context, Niven \cite{N} gives a lower bound of the same character, bounding $\sum_{p \leq L} 1/p$ from below by $\log\log{L}$ less an explicit constant. That argument does not require Chebyshev's estimates for $\psi(x)$. For the correct constant in the asymptotic as $L \rightarrow \infty$, see \cite[Theorem 7, p. 22]{I}.

\section{Proof of Theorem~\ref{thm:mod4}} \label{sec:proof-mod4}

To prove Theorem~\ref{thm:mod4}, we produce a complete bipartite graph joining the permutations of $(2+2i,2,2)$ and $(2-2i,2,2)$. 
One can check as follows that $(m_1 \circ m_2)^{(p-1)/2}$ takes $(2+2i,2,2)$ to $(2,2-2i,2)$. By definition,
\[
m_2 ( 2+2i,2,2) = (2+2i,2+4i,2)
\]
and then
\[
m_1 \circ m_2 (2+2i,2,2) = (2+6i,2+4i,2).
\]
Inductively, one finds that for each $k \geq 0$,
\begin{equation} \label{eqn:inductively}
(m_1 \circ m_2)^k (2+2i,2,2) = (2+(4k+2)i, 2+4ki, 2).
\end{equation}
In particular, with $k = (p-1)/2$, the claim follows since we work modulo $p$:
\[
(m_1 \circ m_2)^{(p-1)/2} (2+2i,2,2) = (2+2p i, 2 + 2(p-1)i, 2) = (2,2-2i,2).
\]
We emphasize that the only coordinate not equal to 2, in addition to moving from the $x$-coordinate to the $y$-coordinate, has changed from $2+2i$ to $2-2i$. 

In the same way, we find that $(m_1 \circ m_3)^{(p-1)/2}$ takes $(2+2i,2,2)$ to $(2,2,2-2i)$, while $(m_2 \circ m_3)^{(p-1)/2}$ takes $(2,2+2i,2)$ to $(2,2,2-2i)$. 
Thus the Markoff graph contains the configuration drawn in Figure~\ref{fig:1mod4}.
We abbreviate $m_j \circ m_k$ by $m_j m_k$.

\begin{figure}
\begin{tikzpicture}[scale=0.75]
\draw (3,0) arc (0:180:3);
\draw (6,0) arc (0:180:6);
\draw (-3,0) arc (180:270:3);
\draw (-6,0) arc (180:270:6);
\draw (0,-3) -- (6,0);
\draw (0,-6) -- (3,0);
\draw (2.3,-2.3) node {$m_2$};
\draw (1.7,0) node {$(2,2,2+2i)$};
\draw (7.3,0) node {$(2,2,2-2i)$};
\draw (3,0) -- (6,0);
\draw (4.5,0.25) node {$m_3$};
\draw (-1.7,0) node {$(2,2+2i,2)$};
\draw (-7.3,0) node {$(2,2-2i,2)$};
\draw (-3,0) -- (-6,0);
\draw (-4.5,0.25) node {$m_2$};
\draw (0,3) -- (0,6);
\draw (0.5,4.5) node {$m_1$};
\draw (0,2.5) node {$(2-2i,2,2)$};
\draw (0,6.5) node {$(2+2i,2,2)$};
\draw (-0.5,-3.4) node {$(2,2-4i,2-2i)$};
\draw (1.8,-6) node {$(2,2+4i,2+2i)$};
\draw (5,5) node {$(m_1  m_3)^{(p-1)/2}$};
\draw (-5,5) node {$(m_1  m_2)^{(p-1)/2}$};
%\draw (-0.5,-1) node {$m_3(m_2 m_3)^{(p-1)/2 - 1}$};
\draw (0,0)++(250:1.5) node {$m_3(m_2 m_3)^{(p-1)/2-1}$};
%arrows
\draw (0,0)++(30:6)--++(90:1);
\draw (0,0)++(30:6)--++(150:1);
\draw (0,0)++(150:6)--++(90:1);
\draw (0,0)++(150:6)--++(30:1);
%attempting more labels
\draw (0,0)++(190:6.5) node {$m_3$};
\draw (0,0)++(200:6.5) node {$m_2$};
\draw (0,0)++(210:6.5) node {$m_3$};
\draw (0,0)++(220:6.5) node[circle,fill,inner sep=0.75pt]{};
\draw (0,0)++(225:6.5) node[circle,fill,inner sep=0.75pt]{};
\draw (0,0)++(230:6.5) node[circle,fill,inner sep=0.75pt]{};
\draw (0,0)++(240:6.5) node {$m_3$};
\draw (0,0)++(250:6.5) node {$m_2$};
\draw (0,0)++(260:6.5) node {$m_3$};
\end{tikzpicture}
\caption{
For any prime $p \equiv 1 \bmod 4$, the Markoff graph contains a subdivision of the complete bipartite graph connecting $(2,2,2+2i)$ and its permutations to $(2,2,2-2i)$ and its permutations. 
}
 \label{fig:1mod4}
\end{figure}
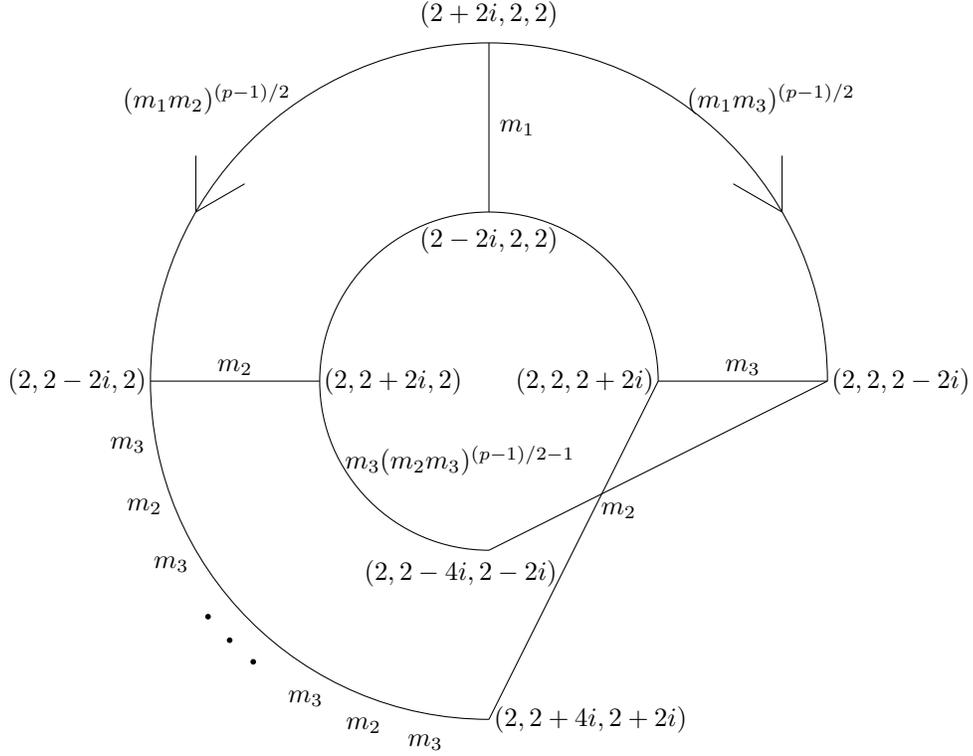
The top half of the figure has an outer curve connecting $(2+2i,2,2)$ to $(2,2-2i,2)$ and $(2,2,2-2i)$ via
\begin{align*}
(m_1 m_3)^{(p-1)/2} (2+2i,2,2) &= (2,2,2-2i) \\
(m_1 m_2)^{(p-1)/2}(2+2i,2,2) &= (2,2-2i,2)
\end{align*}
and an inner curve connecting the points $(2,2+2i,2)$ and $(2,2,2+2i)$ to $(2-2i,2,2)$.
The bottom half shows the analogous relation for the second and third coordinates, namely
\[
(m_2 m_3)^{(p-1)/2} (2,2+2i,2) = (2,2,2-2i),
\]
but because of the choices we have already made in drawing the top half, $(2,2+2i,2)$ is part of the inner circle while $(2,2,2-2i)$ is on the outside. 
Likewise, $(2,2-2i,2)$ on the outside is connected to $(2,2,2+2i)$ on the inside.
To connect the outer and inner circles in this way requires a crossing of edges.
In Figure~\ref{fig:1mod4}, this crossing corresponds to the final move $m_2$ in the paths
\begin{align*}
(2,2,2-2i) &= m_2 (2,2-4i,2-2i) = m_2 \circ m_3 (m_2 m_3)^{(p-1)/2-1} (2,2+2i,2) \\
(2,2,2+2i) &= m_2 (2,2+4i,2+2i) = m_2 (m_3 m_2)^{(p-1)/2} (2,2+2i,2).
\end{align*}

Starting from this configuration, we contract edges as follows to produce a minor isomorphic to $K_{3,3}$.
First, contract the edges forming the inner and outer quarter-circles in the top half of the figure.
This connects $(2+2i,2,2)$ to $(2,2-2i,2)$ and to $(2,2,2-2i)$, as well as $(2-2i,2,2)$ to $(2,2+2i,2)$ and $(2,2,2+2i)$.
Second, contract the edges in the ``outer third quadrant" from $(2,2-2i,2)$ to $(2,2+4i,2+2i)$, leaving a path from $(2,2-2i,2)$ to $(2,2,2+2i)$. 
Third, contract the ``inner third quadrant" to obtain a path from $(2,2+2i,2)$ to $(2,2,2-2i)$. 
The resulting graph has six vertices, namely the permutations of $(2 \pm 2i, 2,2)$, such that all vertices with a coordinate $2+2i$ are connected to all vertices with a coordinate $2-2i$. This is a complete bipartite graph $K_{3,3}$ as required.

In this construction, two moves such as $m_1$ and $m_2$ alternate between the lines $x-y = \pm 2i$ contained in the Markoff surface for $p \equiv 1 \bmod 4$. For $p \equiv 3 \bmod 4$, all the lines are imaginary.

\section{Proof of Theorem~\ref{thm:7}} \label{sec:proof7}

In this section, we prove Theorem~\ref{thm:7} by producing a copy of Figure~\ref{fig:-7square} inside the Markoff graph mod $p$, so long as $-7$ is a quadratic residue.
Consider the Markoff equation with $y=z=-1$. The final coordinate must satisfy
\[
x^2 + 2 =x, \qquad x = \frac{1 \pm \sqrt{-7} }{2}
\]
so there are two solutions $x$ and $1-x$ if $-7$ is a non-zero square modulo $p$; a single solution for $p=7$; and no solutions otherwise. Suppose $-7$ is a square. Then there are six solutions $(x,-1,-1)$, $(1-x,-1,-1)$ and their permutations.
A single move $m_1$ connects $(x,-1,-1)$ to $(1-x,-1,-1)$. A duo of moves $m_1 \circ m_2$ leads to
\[
(x,-1,-1) \mapsto (x,1-x,-1) \mapsto (-1,1-x,-1).
\]
In the same way, $m_1 \circ m_3$ takes $(x,-1,-1)$ to $(-1,-1,1-x)$, and one has similar paths $m_j \circ m_k$ starting from $(-1,x,-1)$ and $(-1,-1,x)$. 
These paths starting from $(x,-1,-1)$ are illustrated in Figure~\ref{fig:-7square}.
Together with their counterparts at $(-1,x,-1)$ and $(-1,-1,x)$, they form a copy of $K_{3,3}$.

\begin{figure}[h]
\begin{tikzpicture}[scale=2]
\pgfmathsetmacro{\w}{1.732}
\draw (0,0) node(xy-){$(x,1-x,-1)$};
\draw (\w,0) node(yx-){$(1-x,x,-1)$};
\draw (xy-)++(30:1) node(--y){$(-1,-1,1-x)$};
\draw (xy-)++(90:1) node(-xy){$(-1,x,1-x)$};
\draw( xy-)++(150:1) node(-x-){$(-1,x,-1)$};
\draw (xy-)++(210:1) node(-y-){$(-1,1-x,-1)$};
\draw (xy-)++(270:1) node(-yx){$(-1,1-x,x)$};
\draw (xy-)++(330:1) node(--x){$(-1,-1,x)$};
\draw (yx-)++(90:1) node(x-y){$(x,-1,1-x)$};
\draw (yx-)++(30:1) node(x--){$(x,-1,-1)$};
\draw (yx-)++(330:1) node(y--){$(1-x,-1,-1)$};
\draw (yx-)++(270:1) node(y-x){$(1-x,-1,x)$};
\draw (--y)--(-xy)--(-x-)--(-y-)--(-yx)--(--x)--(y-x)--(y--)--(x--)--(x-y)--(--y)--(--x);
\draw (-y-)--(xy-)--(x--);
\draw (-x-)--(yx-)--(y--);
\end{tikzpicture}
\caption{If $-7$ is a non-zero square modulo $p$, then $x^2 + y^2 + z^2 = xyz \bmod p$ has solutions $(x,-1,-1)$ and $(1-x,-1,-1)$ and their permutations. From a solution with coordinate $x$, one can reach any solution with coordinate $1-x$ by either a single move or two moves.}
\label{fig:-7square}
\end{figure}
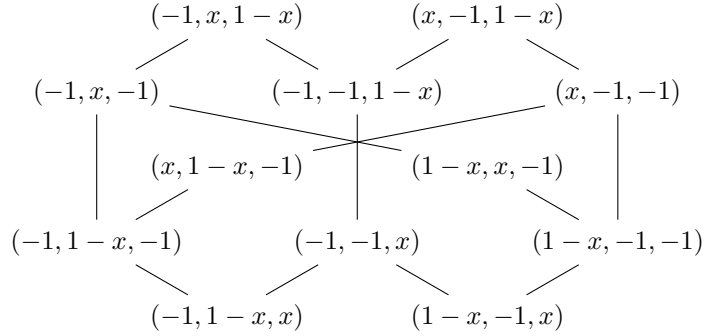

The same configuration occurs in other level sets $x^2 + y^2 + z^2 = xyz + k$ whenever there are solutions with two coordinates equal to $-1$. Setting $y=z=-1$, the solution for $x$ is
\[
x = \frac{1 \pm \sqrt{4k-7} }{2}.
\]
It would seem to follow that the Markoff graph of level $k$ has a non-planar component provided $4k-7$ is a non-zero square modulo $p$. However, the configuration might not be as shown in Figure~\ref{fig:-7square} if there are self-edges. For example, if $k=4$, one has $\sqrt{4k-7}= \sqrt{9}=3$, so the special values are $x=2$ and $1-x=-1$.  Each point of the form $(2,-1,-1)$ has a pair of self-edges because $-1 \mapsto 2(-1)-1 = -1$. 
The whole construction lies in a planar component consisting of $(-1,-1,-1)$ and its neighbours, which is essentially the cluster from Figure~\ref{fig:mod2} since $k=4 \equiv 0 \bmod 2$. 
Self-edges occur in Figure~\ref{fig:-7square} only for $x=-1$ or $x=2$, but these can equal $(1+\sqrt{4k-7})/2$ only for $k=4$.

\section{Planar graphs do not expand} \label{sec:lt}

This section reviews why planar graphs cannot form an expander family, which was our motivation for studying planarity of Markoff graphs (or, more hopefully, their non-planarity).
The failure of expansion in planar graphs is a consequence of a celebrated theorem of Lipton and Tarjan.
\begin{theorem}[Lipton-Tarjan planar separator theorem, \cite{LT}] \label{thm:lipton-tarjan}
For any planar graph on $n$ vertices, the vertex set can be partitioned into three sets $A$, $B$, and $C$ such that no vertex in $A$ is connected to any vertex in $B$, each of $A$ and $B$ contains at most $2n/3$ vertices, and $C$ contains at most $2\sqrt{2n}$ vertices. 
\end{theorem}
Moreover, Lipton and Tarjan give an algorithm for computing such a partition in $O(n)$ steps. 

A standard way to quantify expansion is the Cheeger constant.
For a graph $G$, the Cheeger constant $h(G)$ is defined as
\begin{equation} \label{eqn:cheegerdefn}
h(G) = \min \frac{|\partial A|}{\min(|A|,|G \setminus A|)}
\end{equation}
where the minimum is taken over all non-empty, proper subsets $A$ of the vertices of $G$, and $\partial A$ is the set of edges joining a vertex in $A$ to another vertex in its complement $G \setminus A$. 
If $h(G) = 0$, then $G$ is disconnected since there is a subset $A$ with $|\partial A| = 0$, that is, no edges from $A$ to its complement.
Expansion refers to a sequence of graphs with a growing number of vertices, but $h(G)$ bounded strictly away from 0. 

Theorem~\ref{thm:lipton-tarjan} implies that, for any sequence of planar graphs with a growing number of vertices, $h(G) \rightarrow 0$. Indeed, given a planar graph $G$ on $n$ vertices, consider sets $A$, $B$, and $C$ as in Theorem~\ref{thm:lipton-tarjan}. We use one of the large parts, say $A$, as a candidate for the ratio $|\partial A| \div \min(|A|,|G\setminus A|)$ in the definition (\ref{eqn:cheegerdefn}) of $h(G)$. There are no edges between $A$ and $B$, so
\[
|\partial A | \leq |C| \leq 2\sqrt{2n}.
\]
On the other hand, $|A| \geq n/3 - 2\sqrt{2n}$ because $B$ and $C$ together account for at most $2n/3+2\sqrt{2n}$ vertices. Likewise, $|G \setminus A| \geq n/3$ because $|A| \leq 2n/3$. 
It follows that
\begin{equation} \label{eqn:lt-corollary}
h(G) \leq \frac{ |\partial A| }{\min(|A|, |G \setminus A |) } \leq \frac{2\sqrt{2n}}{n/3 - 2\sqrt{2n}} \lesssim n^{-1/2}
\end{equation}
In particular, $h(G) \rightarrow 0$ as $n \rightarrow \infty$. 

In contrast, numerical evidence \cite{dci-lee} suggests that $h(G)$ is bounded away from 0 for Markoff graphs with $p \rightarrow \infty$. 
It is easier to compute a different measure of expansion, namely the next-largest eigenvalue of the adjacency matrix of $G$. For a $d$-regular graph, the largest eigenvalue is $d$ and we denote the next-largest absolute value among the eigenvalues by $\lambda$. 
The \emph{Cheeger inequality} for $d$-regular graphs (see \cite{AM}, \cite{D}) states that
\begin{equation} \label{eqn:cheegerineq}
\frac{1}{2}(d-\lambda) \leq h(G) \leq \sqrt{2d(d-\lambda)}
\end{equation}
In particular, $h \rightarrow 0$ if and only if $\lambda \rightarrow d$.
The constant function equal to 1 at every vertex is an eigenvector for the eigenvalue $d$. The multiplicity of this eigenvalue is the number of connected components of the graph. This is a practical way to check connectedness of Markoff graphs.

In the Markoff case, $d=3$ and, from the data in \cite{dci-lee}, $\lambda$ appears to converge to different values as $p \rightarrow \infty$ along the subsequences of primes congruent to $1 \bmod 4$ or $3 \bmod 4$. In both cases, $\lambda$ seems to remain bounded away from 3, in which case $h$ must remain bounded away from 0.
Once the number of vertices $n = p^2 \pm 3p$ is large enough, the inequality that would follow from (\ref{eqn:lt-corollary}) and (\ref{eqn:cheegerineq}), namely
\[
\frac{1}{2}(3-\lambda) \leq \frac{2 \sqrt{2n} }{n/3 - 2\sqrt{2n} }
\]
must therefore fail, and then the Markoff graph mod $p$ cannot be planar. 

For example, consider once again $p=19$, the first instance where neither Theorem~\ref{thm:mod4} nor Theorem~\ref{thm:7} applies. The number of vertices in this case is $n=p^2 -3p = 304$, so that
\begin{equation} \label{eqn:948}
\frac{2\sqrt{2n} }{n/3 - 2\sqrt{2n} } = 0.948\ldots
\end{equation}
%0.94804282789334239140519231532285285089
It is feasible to compute all the eigenvalues on a personal computer equipped with Pari \cite{pari}, and the next largest in modulus is approximately
$\lambda = 2.873\ldots$
%2.8731721565992043630806653590072910463
This gives $(3-\lambda)/2 = 0.0634\ldots$ which is well below (\ref{eqn:948}). Thus the spectral method does not apply to $p=19$.
%0.063413921700397818459667320496354476830
Assuming that a similar spectral gap persists for larger primes congruent to 3 mod 4, the comparison would become favourable to deducing non-planarity once $p \geq 163$, at which point $2\sqrt{2n} \div (n/3 - 2\sqrt{2n}) < 0.06$. 
%values of 2sqrt(2n)/(n/3 - 2sqrt(2n))
%[103, 0.091236019325815352446251268917514282945] 
%[107, 0.087473471196692545869268214039447221918]
%[127, 0.072520205822339423997779489230359005577]
%[131, 0.070122797176384841623632588144852041008]
%[139, 0.065774042154987088159061155805072433263]
%[151, 0.060176230534645614762755081617403387283]
%[163, 0.055456545396321861456617937781000627961]
%[167, 0.054043651590525313234907474560755223428]
%[179, 0.050206275865830180158387699200751470370]
%[191, 0.046877733152902035583219254095144742459]
%[199, 0.044893525553114154933189307193764413534]

\section{Examples} \label{sec:examples}

The first non-planar Markoff graph occurs for $p=5$. It is drawn (with crossings) in Figure~\ref{fig:5}. The number of squares is $s=3(p-1)/2=6$, the number of hexagons is $h=p+1=6$, and there are no self-edges. Therefore $V=40$ and $E=3V/2=60$. The construction of Theorem~\ref{thm:mod4} gives cycles of length $2p=10$. By inspection, there are no cycles of length 7, so one can take $g=8$ to improve the bounds on the Euler characteristic.
There are cycles of length 8, for instance traversing a hexagon and one of its adjacent squares, but these do not bound their own faces.  Nevertheless, taking $g=8$ gives $1/2 \leq -\chi$, which rounds to $1 \leq -\chi$.

For a non-orientable surface, formed from a sphere with $n$ cross-caps, the Euler characteristic is $\chi=2-n$, and the bound $1 \leq -\chi$ amounts to $n \geq 3$. For more on cross-caps, see \cite[p. 94--103]{CBG}.
To draw the Markoff graph mod 5 on a surface with $n=3$, imagine a cross-cap attached over each of the three hexagons in Figure~\ref{fig:5}.
For an orientable surface of genus $\gamma$, we would have $\chi = 2-2\gamma$, hence $\gamma \geq 2$, but this bound does not seem to be attainable.

\begin{figure}
\begin{tikzpicture}[scale=1.5]
\draw (0,0) node(333){333};
\draw (333)++(90:1) node(133){133};
\draw (133)++(30:1) node(130){130};
\draw (133)++(150:1) node(103){103};
\draw (130)++(90:1) node(120){120};
\draw (130)++(0:1) node(430){430};
\draw (120)++(0:1) node(420){420};
\draw (420)++(45:1) node(423){423};
\draw (103)++(90:1) node(102){102};
\draw (103)++(180:1) node(403){403};
\draw (102)++(180:1) node(402){402};
\draw (402)++(135:1) node(432){432};
\draw (120)++(150:1) node(122){122};
\draw (122)++(90:1) node(322){322};
\draw (333)++(210:1) node(313){313};
\draw (313)++(270:1) node(310){310};
\draw (313)++(150:1) node(013){013};
\draw (013)++(120:1) node(043){043};
\draw (013)++(210:1) node(012){012};
\draw (310)++(210:1) node(210){210};
\draw (210)++(150:1) node(212){212};
\draw (212)++(210:1) node(232){232};
\draw (012)++(120:1) node(042){042};
\draw (210)++(300:1) node(240){240};
\draw (310)++(300:1) node(340){340};
\draw (042)++(165:1) node(342){342};
\draw (240)++(255:1) node(243){243}; 
\draw (333)++(330:1) node(331){331};
\draw (331)++(270:1) node(301){301};
\draw (331)++(30:1) node(031){031};
\draw (031)++(330:1) node(021){021};
\draw (301)++(330:1) node(201){201};
\draw (201)++(30:1) node(221){221};
\draw (221)++(330:1) node(223){223};
\draw (301)++(240:1) node(304){304};
\draw (201)++(240:1) node(204){204};
\draw (031)++(60:1) node(034){034}; 
\draw (021)++(60:1) node(024){024};
\draw (024)++(15:1) node(324){324};
\draw (204)++(285:1) node(234){234}; 
\draw[very thick] (322)--(122);
\draw[very thick] (232)--(212);
\draw[very thick] (223)--(221);
\draw[very thick] (122)--(102)--(402)--(432)--(232);
\draw[very thick] (232)--(234)--(204)--(201)--(221)--(021)--(024)--(324);
\draw[very thick] (324)--(322);
\draw[very thick] (322)--(342);
\draw[very thick] (342)--(042)--(012)--(212)--(210)--(240)--(243);
\draw[very thick] (243)--(223);
\draw[very thick] (122)--(120)--(420)--(423);
\draw[very thick] (423)--(223);
\draw[dashed] (333)--(133)--(130)--(120);
\draw[dashed] (130)--(430)--(420);
\draw[dashed] (133)--(103)--(403)--(402);
\draw[dashed] (103)--(102);
\draw[dashed] (432)--(430);
\draw[dashed] (423)--(403);
\draw[dashed] (042)--(043)--(013)--(313)--(333)--(331)--(031)--(034)--(024);
\draw[dashed] (012)--(013);
\draw[dashed] (210)--(310)--(313);
\draw[dashed] (240)--(340)--(310);
\draw[dashed] (204)--(304)--(301)--(331);
\draw[dashed] (301)--(201);
\draw[dashed] (031)--(021);
\draw[dashed] (342)--(340);
\draw[dashed] (243)--(043);
\draw[dashed] (234)--(034);
\draw[dashed] (324)--(304);
\end{tikzpicture}
\caption{The Markoff graph mod 5. The 40 vertices are labelled $xyz$, where $x^2+y^2+z^2 = xyz \bmod 5$. The thicker edges illustrate the construction proving Theorem~\ref{thm:mod4} for $p=5$. One can choose $i=\pm 2$ and have $i^2 \equiv -1 \bmod 5$. Points of the form $(3,2,2)$ and $(1,2,2)$ play the role of $(2 \pm 2i, 2, 2)$ from Figure~\ref{fig:1mod4}}
\label{fig:5}
\end{figure}
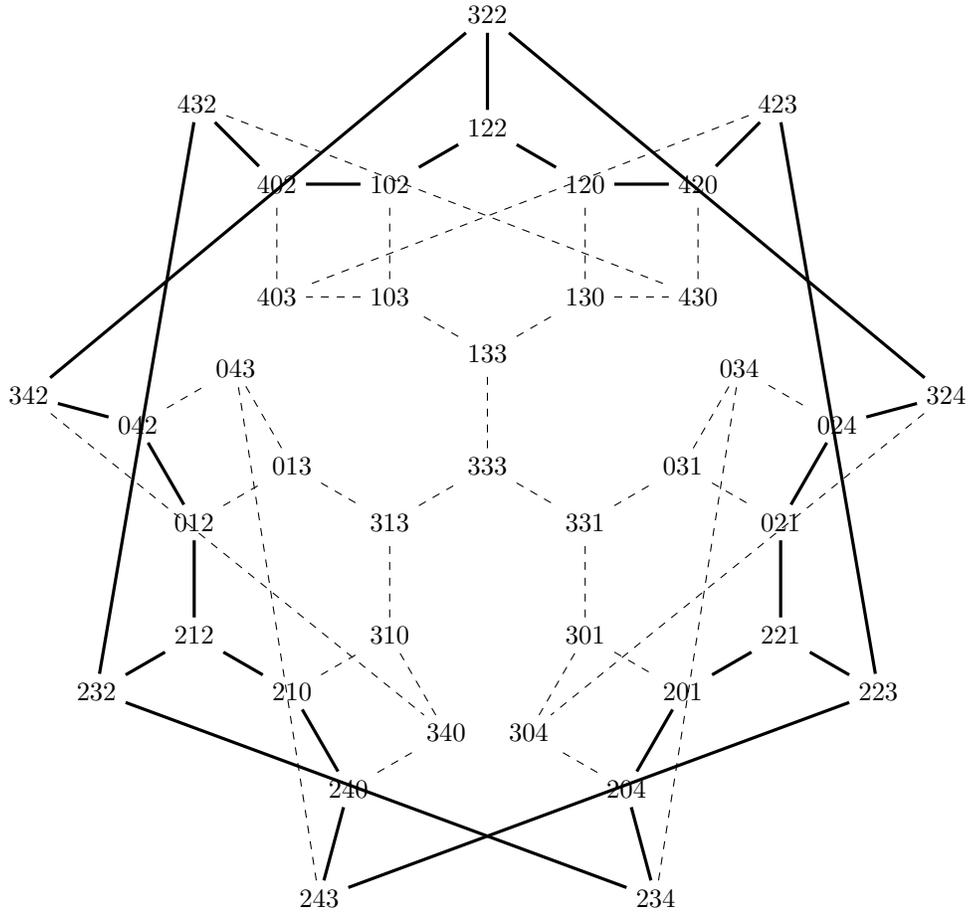

The example $p=11$ illustrates both Theorem~\ref{thm:7} and Proposition~\ref{prop:323121}. Carlitz's formula $p^2-3p$ (Lemma~\ref{lem:counts}) gives 88 vertices in total, which can be thought of as four signed copies of $22=16+6$. The 16 vertices in this partition form a tree following three steps from any of $(3,3,3)$ or its sign changes such as (8,8,3). The self-edges from Proposition~\ref{prop:323121} occur at $(3,4,6)$, as well as its permutations and sign changes. This limits the branching so that there are only 16 vertices per tree. The remaining 6 vertices in $4\times (16+6)$ come from four copies of Figure~\ref{fig:-7square}, one for each sign change. We have $\sqrt{-7}=\sqrt{4}=2 \bmod 11$, and $1/2 = 6$, so the special values $\frac{1}{2} (1 \pm \sqrt{-7})$ are $x=5$ and $1-x=7$. 

The cage for $p=11$ consists of triples with a coordinate equal to $\pm 5$. Indeed, since $z=0, \pm 2$ do not occur for $p \equiv 3 \bmod 4$, the possible values are $\pm z = 1, 3, 4, 5$. Both signs lead to the same order for $\zeta^2$, where $z = \zeta + \zeta^{-1}$. For $z = 5$, writing $i^2=-1$, we have
\begin{align*}
\zeta &= \frac{z + \sqrt{z^2-4}}{2} = \frac{5+\sqrt{-1}}{2} = 3(-1+2i) \\
\zeta^2 &= 6+8i \\
\zeta^4 &= 5+8i \\
\zeta^8 &= 5-8i = \overline{\zeta^4}
\end{align*}
From these values, it follows that $\zeta^{12}=1$ and no smaller exponent works. Thus $\zeta$ has order $12 = p+1$, which is as large as possible, putting $z=5$ in the cage.

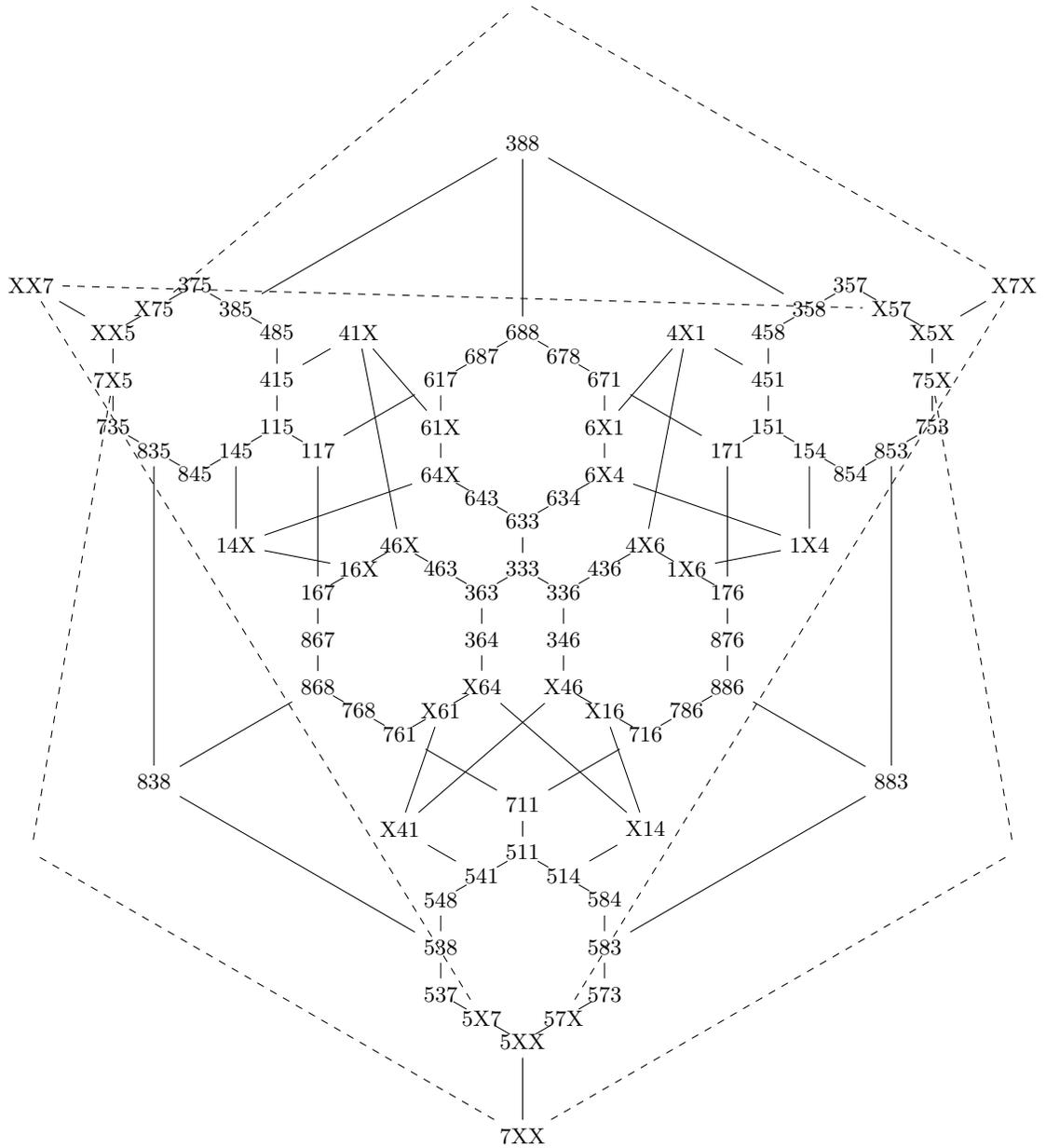
\begin{figure}[t]
\begin{tikzpicture}[scale=0.675]
\pgfmathsetmacro{\h}{5}
\pgfmathsetmacro{\a}{2}
\pgfmathsetmacro{\b}{2}
\pgfmathsetmacro{\c}{\h+5+\b}
\draw (0,0) node(333){\small{333}};
\draw (333)++(90:1) node(633){\small{633}};
\draw (333)++(210:1) node(363){\small{363}};
\draw (333)++(330:1) node(336){\small{336}};
\draw (633)++(150:1) node(643){\small{643}};
\draw (633)++(30:1) node(634){\small{634}};
\draw (634)++(30:1) node(6x4){\small{6X4}};
\draw (643)++(150:1) node(64x){\small{64X}};
\draw (6x4)++(90:1) node(6x1){\small{6X1}};
\draw (6x1)++(90:1) node(671){\small{671}};
\draw (64x)++(90:1) node(61x){\small{61X}};
\draw (61x)++(90:1) node(617){\small{617}};
\draw (671)++(150:1) node(678){\small{678}};
\draw (678)++(150:1) node(688){\small{688}};
\draw (617)++(30:1) node(687){\small{687}};
\draw (363)++(150:1) node(463){\small{463}};
\draw (463)++(150:1) node(46x){\small{46X}};
\draw (363)++(270:1) node(364){\small{364}};
\draw (364)++(270:1) node(x64){\small{X64}};
\draw (46x)++(210:1) node(16x){\small{16X}};
\draw (16x)++(210:1) node(167){\small{167}};
\draw (167)++(270:1) node(867){\small{867}};
\draw (867)++(270:1) node(868){\small{868}};
\draw (868)++(330:1) node(768){\small{768}};
\draw (768)++(330:1) node(761){\small{761}};
\draw (x64)++(210:1) node(x61){\small{X61}};
\draw (336)++(30:1) node(436){\small{436}};
\draw (436)++(30:1) node(4x6){\small{4X6}};
\draw (336)++(270:1) node(346){\small{346}};
\draw (346)++(270:1) node(x46){\small{X46}};
\draw (4x6)++(330:1) node(1x6){\small{1X6}};
\draw (1x6)++(330:1) node(176){\small{176}};
\draw (176)++(270:1) node(876){\small{876}};
\draw (876)++(270:1) node(886){\small{886}};
\draw (886)++(210:1) node(786){\small{786}};
\draw (x46)++(330:1) node(x16){\small{X16}};
\draw (x16)++(330:1) node(716){\small{716}};
\draw (688)++(90:4) node(388){\small{388}};
\draw (868)++(210:4) node(838){\small{838}};
\draw (886)++(330:4) node(883){\small{883}};
\draw (333)++(270:\h) node(711){\small{711}};
\draw (711)++(270:1) node(511){\small{511}};
\draw (333)++(150:\h) node(117){\small{117}};
\draw (117)++(150:1) node(115){\small{115}};
\draw (333)++(30:\h) node(171){\small{171}};
\draw (171)++(30:1) node(151){\small{151}};
\draw (115)++(90:1) node(415){\small{415}};
\draw (415)++(90:1) node(485){\small{485}};
\draw (115)++(210:1) node(145){\small{145}};
\draw (145)++(210:1) node(845){\small{845}};
\draw (485)++(150:1) node(385){\small{385}};
\draw (385)++(150:1) node(375){\small{375}};
\draw (845)++(150:1) node(835){\small{835}};
\draw (835)++(150:1) node(735){\small{735}};
\draw (735)++(90:1) node(7x5){\small{7X5}};
\draw (7x5)++(90:1) node(xx5){\small{XX5}};
\draw (xx5)++(30:1) node(x75){\small{X75}};
\draw (xx5)++(150:\b) node(xx7){\small{XX7}};
\draw (151)++(90:1) node(451){\small{451}};
\draw (451)++(90:1) node(458){\small{458}};
\draw (458)++(30:1) node(358){\small{358}};
\draw (358)++(30:1) node(357){\small{357}};
\draw (357)++(330:1) node(x57){\small{X57}};
\draw (x57)++(330:1) node(x5x){\small{X5X}};
\draw (x5x)++(270:1) node(75x){\small{75X}};
\draw (75x)++(270:1) node(753){\small{753}};
\draw (753)++(210:1) node(853){\small{853}};
\draw (853)++(210:1) node(854){\small{854}};
\draw (854)++(150:1) node(154){\small{154}};
\draw (x5x)++(30:\b) node(x7x){\small{X7X}};
\draw (511)++(330:1) node(514){\small{514}};
\draw (514)++(330:1) node(584){\small{584}};
\draw (584)++(270:1) node(583){\small{583}};
\draw (583)++(270:1) node(573){\small{573}};
\draw (573)++(210:1) node(57x){\small{57X}};
\draw (57x)++(210:1) node(5xx){\small{5XX}};
\draw (5xx)++(150:1) node(5x7){\small{5X7}};
\draw (5x7)++(150:1) node(537){\small{537}};
\draw (537)++(90:1) node(538){\small{538}};
\draw (538)++(90:1) node(548){\small{548}};
\draw (548)++(30:1) node(541){\small{541}};
\draw (5xx)++(270:\b) node(7xx){\small{7XX}};
\draw (415)++(30:\a) node(41x){\small{41X}};
\draw (145)++(270:\a) node(14x){\small{14X}};
\draw (541)++(150:\a) node(x41){\small{X41}};
\draw (514)++(30:\a) node(x14){\small{X14}};
\draw (154)++(270:\a) node(1x4){\small{1X4}};
\draw (451)++(150:\a) node(4x1){\small{4X1}};
\draw (385)--(388)--(358);
\draw (835)--(838)--(538);
\draw (583)--(883)--(853);
\draw (167)--(117)--(617);
\draw (671)--(171)--(176);
\draw (716)--(711)--(761);
\draw (388)--(688);
\draw (838)--(868);
\draw (883)--(886);
\draw[dashed] (xx7)--(5x7);
\draw[dashed] (xx7)--(x57);
\draw[dashed] (x7x)--(57x);
\draw (333)++(90:\c) node(c88){};
\draw (333)++(330:\c) node(88c){};
\draw (333)++(210:\c) node(8c8){};
\draw[dashed] (x75)--(c88)--(x7x);
\draw[dashed] (75x)--(88c)--(7xx);
\draw[dashed] (7x5)--(8c8)--(7xx);
\draw (145)--(14x)--(16x);
\draw (14x)--(64x);
\draw (415)--(41x)--(61x);
\draw (41x)--(46x);
\draw (451)--(4x1)--(6x1);
\draw (4x1)--(4x6);
\draw (154)--(1x4)--(1x6);
\draw (1x4)--(6x4);
\draw (514)--(x14)--(x16);
\draw (x14)--(x64);
\draw (541)--(x41)--(x61);
\draw (x41)--(x46);
\draw (688)--(678)--(671)--(6x1)--(6x4)--(634)--(633)--(643)--(64x)--(61x)--(617)--(687)--(688);
\draw (868)--(867)--(167)--(16x)--(46x)--(463)--(363)--(364)--(x64)--(x61)--(761)--(768)--(868);
\draw (886)--(786)--(716)--(x16)--(x46)--(346)--(336)--(436)--(4x6)--(1x6)--(176)--(876)--(886);
\draw (633)--(333)--(363);
\draw (x5x)--(75x)--(753)--(853)--(854)--(154)--(151)--(451)--(458)--(358)--(357)--(x57)--(x5x);
\draw (511)--(514)--(584)--(583)--(573)--(57x)--(5xx)--(5x7)--(537)--(538)--(548)--(541)--(511);
\draw (115)--(145)--(845)--(835)--(735)--(7x5)--(xx5)--(x75)--(375)--(385)--(485)--(415)--(115);
\draw (333)--(336);
\draw (115)--(117);
\draw (xx5)--(xx7);
\draw (x5x)--(x7x);
\draw (511)--(711);
\draw (5xx)--(7xx);
\draw (171)--(151);
\end{tikzpicture}
\caption{The Markoff graph mod 11. Labels $xyz$ abbreviate $(x,y,z)$, and $X$ denotes $10=-1 \bmod 11$. The dashed edges outline a copy of Figure~\ref{fig:-7square}. The cage consists of triples with some coordinate equal to 5 or 6. Self-edges occur at permutations and sign changes of (3,4,6), for instance (8,7,6) or (3,7,5).}
\label{fig:11}
\end{figure}

The first case not covered by either Theorem~\ref{thm:mod4} or Theorem~\ref{thm:7} is $p=19$, since $-7 \equiv 2^2 \bmod 11$ and the other primes $5 \leq p \leq 17$ are congruent to $1 \bmod 4$.
Figure~\ref{fig:19} exhibits a copy of $K_{3,3}$ showing that the Markoff graph mod 19 is not planar, as we already know from (\ref{eqn:7mod12}).
The various paths in Figure~\ref{fig:19} were found by trial and error after drawing part of the graph. 

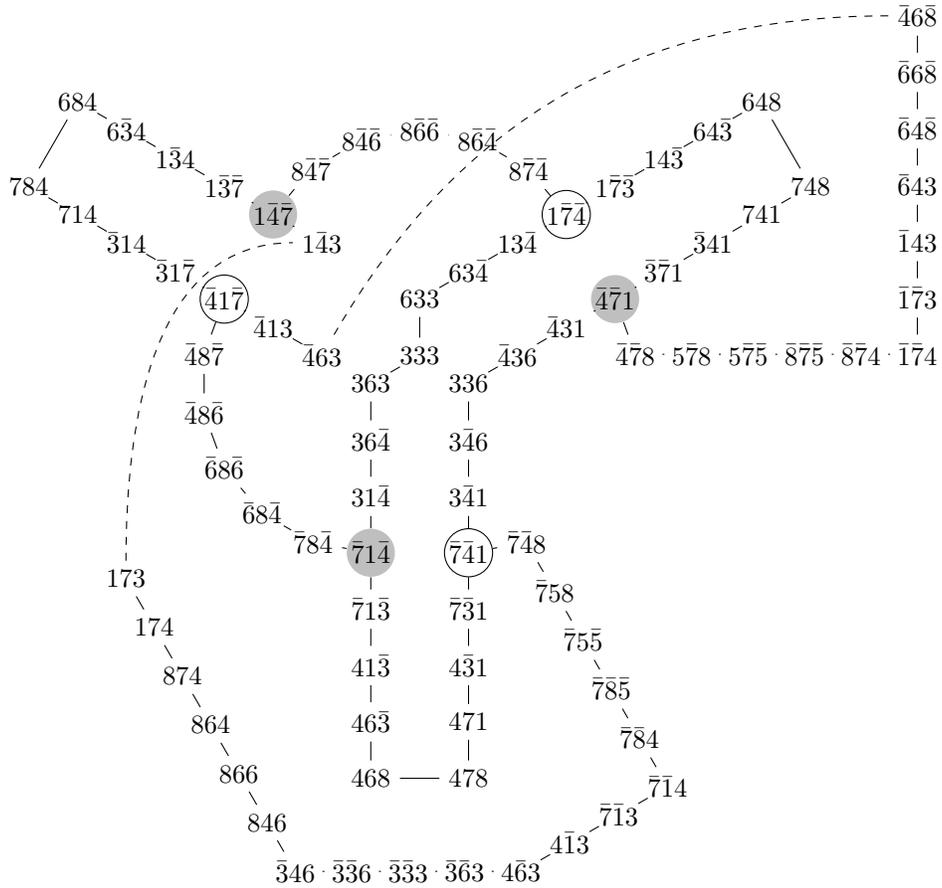
\begin{figure}[t]
\center
\begin{tikzpicture}[scale=0.75]
\draw (0,0) node(333){333};
\draw (90:1) node(633){633};
\draw (210:1) node(363){363};
\draw (330:1) node(336){336};
\draw (633)++(30:1) node(63-4){$63\bar{4}$};
\draw (633)++(150:1) node(6-43){};%{$6\bar{4}3$};
\draw (363)++(270:1) node(36-4){$36\bar{4}$};
\draw (363)++(150:1) node(-463){$\bar{4}63$};
\draw (336)++(270:1) node(3-46){$3\bar{4}6$};
\draw (336)++(30:1) node(-436){$\bar{4}36$};
\draw (3-46)++(270:1) node(3-41){$3\bar{4}1$};
\draw (3-41)++(270:1) node[circle,draw,inner sep=0.5](-7-41){$\bar{7}\bar{4}1$};
\draw (-436)++(30:1) node(-431){$\bar{4}31$};
\draw (-431)++(30:1) node[circle,fill=lightgray, inner sep=0.5](-4-71){$\bar{4}\bar{7}1$};
\draw (63-4)++(30:1) node(13-4){$13\bar{4}$};
\draw (13-4)++(30:1) node[circle,draw,inner sep=0.5](1-7-4){$1\bar{7}\bar{4}$};
\draw (36-4)++(270:1) node(31-4){$31\bar{4}$};
\draw (31-4)++(270:1) node[circle,fill=lightgray, inner sep=0.5](-71-4){$\bar{7}1\bar{4}$};
\draw (-463)++(150:1) node(-413){$\bar{4}13$};
\draw (-413)++(150:1) node[circle,draw,inner sep=0.5](-41-7){$\bar{4}1\bar{7}$};
\draw (6-43)++(150:1) node(1-43){$1\bar{4}3$};
\draw (1-43)++(150:1) node[circle,fill=lightgray, inner sep=0.5](1-4-7){$1\bar{4}\bar{7}$};
\draw (633)++(130:3) node(8-4-7){$8\bar{4}\bar{7}$};
\draw (633)++(110:3) node(8-4-6){$8\bar{4}\bar{6}$};
\draw (633)++(90:3) node(8-6-6){$8\bar{6}\bar{6}$};
\draw (633)++(70:3) node(8-6-4){$8\bar{6}\bar{4}$};
\draw (633)++(50:3) node(8-7-4){$8\bar{7}\bar{4}$};
\draw (363)++(170:3) node(-48-7){$\bar{4}8\bar{7}$};
\draw (363)++(190:3) node(-48-6){$\bar{4}8\bar{6}$};
\draw (363)++(210:3) node(-68-6){$\bar{6}8\bar{6}$};
\draw (363)++(230:3) node(-68-4){$\bar{6}8\bar{4}$};
\draw (363)++(250:3) node(-78-4){$\bar{7}8\bar{4}$};
\draw (-4-71)++(30:1) node(-3-71){$\bar{3}\bar{7}1$};
\draw (1-7-4)++(30:1) node(1-7-3){$1\bar{7}\bar{3}$};
\draw (-3-71)++(30:1) node(-341){$\bar{3}41$};
\draw (1-7-3)++(30:1) node(14-3){$14\bar{3}$};
\draw (-341)++(30:1) node(741){741};
\draw (14-3)++(30:1) node(64-3){$64\bar{3}$};
\draw (741)++(30:1) node(748){748};
\draw (64-3)++(30:1) node(648){648};
\draw (-41-7)++(150:1) node(-31-7){$\bar{3}1\bar{7}$};
\draw (1-4-7)++(150:1) node(1-3-7){$1\bar{3}\bar{7}$};
\draw (-31-7)++(150:1) node(-314){$\bar{3}14$};
\draw (1-3-7)++(150:1) node(1-34){$1\bar{3}4$};
\draw (-314)++(150:1) node(714){714};
\draw (1-34)++(150:1) node(6-34){$6\bar{3}4$};
\draw (714)++(150:1) node(784){784};
\draw (6-34)++(150:1) node(684){684};
\draw (-71-4)++(270:1) node(-71-3){$\bar{7}1\bar{3}$};
\draw (-7-41)++(270:1) node(-7-31){$\bar{7}\bar{3}1$};
\draw (-71-3)++(270:1) node(41-3){$41\bar{3}$};
\draw (-7-31)++(270:1) node(4-31){$4\bar{3}1$};
\draw (41-3)++(270:1) node(46-3){$46\bar{3}$};
\draw (4-31)++(270:1) node(471){$471$};
\draw (46-3)++(270:1) node(468){468};
\draw (471)++(270:1) node(478){478};
\draw (336)++(290:3) node(-7-48){$\bar{7}\bar{4}8$};
\draw (336)++(10:3) node(-4-78){$\bar{4}\bar{7}8$};
\draw (-4-78)++(0:1) node(5-78){$5\bar{7}8$};
\draw (5-78)++(0:1) node(5-7-5){$5\bar{7}\bar{5}$};
\draw (5-7-5)++(0:1) node(-8-7-5){$\bar{8}\bar{7}\bar{5}$};
\draw (-8-7-5)++(0:1) node(-8-74){$\bar{8}\bar{7}4$};
\draw (-8-74)++(0:1) node(-1-74){$\bar{1}\bar{7}4$};
%try right-angle turn
\draw (-1-74)++(90:1) node(-1-73){$\bar{1}\bar{7}3$};
\draw (-1-73)++(90:1) node(-143){$\bar{1}43$};
\draw (-143)++(90:1) node(-643){$\bar{6}43$};
%guess angle 300
\draw (-7-48)++(300:1) node(-758){$\bar{7}58$};
\draw (-758)++(300:1) node(-75-5){$\bar{7}5\bar{5}$};
\draw (-75-5)++(300:1) node(-7-8-5){$\bar{7}\bar{8}\bar{5}$};
\draw (-7-8-5)++(300:1) node(-7-84){$\bar{7}\bar{8}4$};
\draw (-7-84)++(300:1) node(-7-14){$\bar{7}\bar{1}4$};
%try right-angle turn from guessed 300
\draw (-7-14)++(210:1) node(-7-13){$\bar{7}\bar{1}3$};
\draw (-7-13)++(210:1) node(4-13){$4\bar{1}3$};
\draw (4-13)++(210:1) node(4-63){$4\bar{6}3$};
%now the symmetry is broken
\draw (-643)++(90:1) node(-64-8){$\bar{6}4\bar{8}$};
\draw (-64-8)++(90:1) node(-66-8){$\bar{6}6\bar{8}$};
\draw (-66-8)++(90:1) node(-46-8){$\bar{4}6\bar{8}$};
\draw (4-63)++(180:1) node(-3-63){$\bar{3}\bar{6}3$};
\draw (-3-63)++(180:1) node(-3-33){$\bar{3}\bar{3}3$};
\draw (-3-33)++(180:1) node(-3-36){$\bar{3}\bar{3}6$};
\draw (-3-36)++(180:1) node(-346){$\bar{3}46$};
%try angle 120
\draw (-346)++(120:1) node(846){846};
\draw (846)++(120:1) node(866){866};
\draw (866)++(120:1) node(864){864};
\draw (864)++(120:1) node(874){874};
\draw (874)++(120:1) node(174){174};
\draw (174)++(120:1) node(173){173};
\draw[dashed] (-46-8) to[bend right] (-463);
\draw[dashed] (173) to[out=90,in=180] (1-43);
%y shape
\draw (-71-4)--(-71-3)--(41-3)--(46-3)--(468)--(478)--(471)--(4-31)--(-7-31)--(-7-41);
\draw (1-4-7)--(1-3-7)--(1-34)--(6-34)--(684)--(784)--(714)--(-314)--(-31-7)--(-41-7);
\draw (-4-71)--(-3-71)--(-341)--(741)--(748)--(648)--(64-3)--(14-3)--(1-7-3)--(1-7-4);
%arcs
\draw (1-4-7)--(8-4-7)--(8-4-6)--(8-6-6)--(8-6-4)--(8-7-4)--(1-7-4);
\draw (-41-7)--(-48-7)--(-48-6)--(-68-6)--(-68-4)--(-78-4)--(-71-4);
%middle paths
\draw (-71-4)--(31-4)--(36-4)--(363)--(333)--(633)--(63-4)--(13-4)--(1-7-4);
\draw (-4-71)--(-431)--(-436)--(336)--(3-46)--(3-41)--(-7-41);
%finish path from (-4-71) to (-41-7)
\draw (-4-71)--(-4-78)--(5-78)--(5-7-5)--(-8-7-5)--(-8-74)--(-1-74);
\draw (-1-74)--(-1-73)--(-143)--(-643)--(-64-8)--(-66-8)--(-46-8);
\draw (-463)--(-413)--(-41-7);
%path from (1-4-7) to (-7-41)
\draw (1-4-7)--(1-43);
\draw (173)--(174)--(874)--(864)--(866)--(846)--(-346)--(-3-36)--(-3-33);
\draw (-3-33)--(-3-63)--(4-63);
\draw (4-63)--(4-13)--(-7-13)--(-7-14)--(-7-84)--(-7-8-5)--(-75-5)--(-758)--(-7-48)--(-7-41);
\end{tikzpicture}
\caption{
The Markoff graph modulo 19 contains a subdivision of the complete bipartite graph joining the even permutations of $(1,-4,-7)$, circled in grey, to the odd permutations, circled in white. Triples $(x,y,z)$ are abbreviated as $xyz$, and $\bar{x}$ denotes $-x \bmod 19$. 
}
 \label{fig:19}
\end{figure}

\section{Conclusion} \label{sec:conc}

We have shown that the Markoff graph mod 7 is the last of its kind: for $p \neq 2, 3, 7$, these graphs are not planar. Moreover, the Euler characteristic of a surface in which they can be embedded is, in absolute value, at least roughly $p^2/2$ as $p \rightarrow \infty$. 
This non-planarity is consistent with the conjecture that the Markoff graphs form an expander family as $p \rightarrow \infty$.
For $p$ in various arithmetic progressions, non-planarity can be seen by explicit constructions involving $\sqrt{-1}$ or $\sqrt{-7}$.
The general argument is based on variations of the classical Lemma~\ref{lem:planar}. These can be thought of either as an upper bound for the number of vertices $V$ of a 3-regular graph embedded in a surface of given Euler characteristic $\chi$, or as a bound for $\chi$ given $V$. The methods are applicable to other examples beyond the Markoff graphs mod $p$, as long as we have some knowledge of the short cycles.

Lemma~\ref{lem:planar} has some interesting sharp cases for graphs of higher degree. If every vertex has degree $d$, then $E = dV/2$. We can always take $g=3$ for graphs without repeated edges. For a planar graph, Euler's formula then implies $E \leq g(V-2)/(g-2)$, or
\[
V \geq \frac{12}{6-d}
\]
This is achieved for $d=3$ by the tetrahedron with $V=4$; for $d=4$ by the octahedron with $V=6$; and for $d=5$ by the icosahedron with $V=12$. 
All of the Markoff graphs have the symmetries of a tetrahedron, both rotations and reflections. These act by the four sign changes such as $(x,y,z) \mapsto (-x,-y,z)$, together with permutations of the coordinates. 
The Markoff moves themselves give further symmetries, which are closely related to the projective linear group $\PGL(2,p)$. Especially for small primes such as $p=5,7,11$, there might be good ways to draw the Markoff graphs on Platonic solids with cross-caps or handles attached. 

For example, Figure~\ref{fig:mod7} could be folded into a tetrahedron with $(3,3,3)$ or one of its sign changes as the center of each face, or as vertices. One can also see the outermost hexagon in Figure~\ref{fig:mod7} or \ref{fig:11} as a cross-section of a cube.
Another natural home for the Markoff graphs mod $p$, in view of $\PGL(2,p)$, would be the hyperbolic surfaces defined from related subgroups of $\PGL(2,\R)$, or perhaps 3-dimensional hyperbolic models with respect to $\PGL(2,\C)$, or the non-congruence modular curves from \cite{C}.
It would be interesting to compare Theorem~\ref{thm:euler-char} with embeddings of minimal complexity, and find the optimal $c$ for which there are sequences of embeddings with $-\chi=(c+o(1))p^2$ as $p \rightarrow \infty$.

Finally, we comment on another scaling of the Markoff equation, which is the usual form over the integers:
%\[
$
x^2+y^2+z^2=3xyz
$.
%\]
Multiplying each variable by 3 transforms this to the Markoff graphs studied here. The moves are scaled in a compatible way, for instance 
\[
x \mapsto 3yz-x = \frac{(3y)(3z)-3x}{3}. 
\]
For $p \neq 3$, the scaling is invertible and either equation leads to the same Markoff graph mod $p$. For $p=3$, the $3xyz$-version of the Markoff equation reduces to $x^2+y^2+z^2=0$, where the cubic term has disappeared. Whereas $(0,0,0)$ is the only solution to $x^2+y^2+z^2=xyz$, this form has 8 other solutions $(\pm 1, \pm 1, \pm 1)$. The moves collapse to sign changes: $x \mapsto 3yz-x = -x \bmod 3$. In this way, the rescaled Markoff graph mod 3 can be drawn as a cube, giving a more interesting planar example than the empty graph we dismissed earlier.
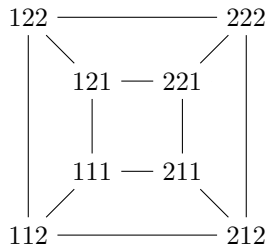
\begin{center}
\begin{tikzpicture}[scale=1.2]
\draw (0,0) node(+++){$111$};
\draw (1,0) node(-++){$211$};
\draw (0,1) node(+-+){$121$};
\draw (1,1) node(--+){$221$};
\draw (--+)++(45:1) node(---){$222$};
\draw (+-+)++(135:1) node(+--){$122$};
\draw (+++)++(225:1) node(++-){$112$};
\draw (-++)++(315:1) node(-+-){$212$};
\draw (+++)--(+-+)--(--+)--(--+)--(---)--(+--)--(++-)--(-+-)--(-++);
\draw (---)--(-+-);
\draw (--+)--(-++)--(+++);
\draw (+--)--(+-+);
\draw (++-)--(+++);
\end{tikzpicture}
\end{center}

\section*{Acknowledgements}

We thank Peter Sarnak, Elena Fuchs, Michael Magee, Eva Bayer-Fluckiger, Martin Stoller, and Maryna Viazovska for their encouragement in this project.
We thank Will Sawin for suggesting another approach in case the graph is disconnected (apply the arguments of Section~\ref{sec:euler} to each component separately and conclude at least one of them is non-planar, though not necessarily the component containing the cage), and Will Chen for suggesting that the surfaces from \cite{C} might give embeddings with $-\chi$ asymptotically as small as possible.
We are very grateful to the anonymous referees for their feedback, including the possibility of $K_5$ occurring as a graph minor.
Many thanks also to Tim Browning for advice on the manuscript.

\end{document}